\documentclass[hidelinks,onefignum,onetabnum]{siamart251216}

\usepackage{lipsum}
\usepackage{amsfonts}
\usepackage{graphicx}
\usepackage{epstopdf}
\usepackage{booktabs,array}
\usepackage{algorithmic}
\ifpdf
  \DeclareGraphicsExtensions{.eps,.pdf,.png,.jpg}
\else
  \DeclareGraphicsExtensions{.eps}
\fi

\newsiamremark{remark}{Remark}
\newsiamremark{conjecture}{Conjecture}
\newsiamremark{hypothesis}{Hypothesis}
\crefname{hypothesis}{Hypothesis}{Hypotheses}
\newsiamthm{claim}{Claim}
\newsiamremark{fact}{Fact}
\crefname{fact}{Fact}{Facts}

\headers{p-Laplacian diffusion}{W. Zhang, H. Asgaribakhtiari, A. Pawar and R. Parshad}

\title{Novel Dynamics in Models of Angiogenesis with p-Laplacian diffusion \thanks{Submitted to the editors 08/07/2026.
\funding{This work was partially funded by National Science Foundation (Grant No. DMS-2533961)}}}

\author{Wenbo Zhang\thanks{Department of Mathematics, Iowa State University, Ames, IA 50011 USA (\email{wbzhang@iastate.edu}).}
\and Hossein Asgaribakhtiari\thanks{Department of Mechanical Engineering, Iowa State University, Ames, IA 50011 USA
  (\email{asgari@iastate.edu}).}
  \and Aishwarya Pawar\thanks{Department of Mechanical Engineering, Iowa State University, Ames, IA 50011 USA
  (\email{arpawar@iastate.edu}).}
\and Rana D. Parshad\thanks{Department of Mathematics, Iowa State University, Ames, IA 50011
  USA
  (\email{rparshad@iastate.edu}).}
}
\usepackage{amsopn}

\ifpdf
\hypersetup{
  pdftitle={Novel Dynamics in Models of Angiogenesis with p-Laplacian diffusion},
  pdfauthor={Wenbo Zhang, Hossein Asgaribakhtiari, Aishwarya Pawar and Rana D. Parshad}
}

\begin{document}

\maketitle

\begin{abstract}
Ischemic heart diseases represent the leading cause of  mortality worldwide. Revascularization, the process to restore blood flow in blockages, shows promise. To this end, mathematical models for angiogenesis, the process by which new blood vessels form from existing ones, have been extremely well investigated. In the current work, we consider a classical two species model for angiogenesis, consisting of cell and VEGF populations. However, we assume the cells move according to p-Laplacian diffusion, which could be both ``fast" ($1<p<2$) and ``slow" ($p>2$), in addition to normal diffusion ($p=2$). We first show that the system is well posed in a weak sense when $p>\frac{3}{2}$, for sufficiently small initial data. Next, we show that the p-Laplacian can lead to several novel dynamics not reported earlier; these include increased cellular proliferation via bi-modal and multi spike solutions, gain of regularity, prevention of finite time blow-up, cell depletion via finite time extinction, and Turing patterns. We discuss applications of these results for cardiac health via a digital twins framework.
\end{abstract}

\begin{keywords}
angiogenesis, p-Laplacian, weak solution, signal dependent chemotaxis, cardiac health
\end{keywords}

\begin{MSCcodes}
35D30, 35K92, 92C17, 92C50
\end{MSCcodes}

\section{Introduction}\label{intro}
Angiogenesis is the process of the formation of new capillaries from pre-existing vasculature, essential for several physiological processes such as wound healing, embryonic development, and tissue regeneration. While revascularization shows promise as an effective therapeutic method for diseases such as ischemic heart diseases \cite{shepherd2004rapid}, the methods lack the predictability and control required for effective vascular formation. Therapeutic angiogenesis can lead to immature or leaky vessels that lack perfusion capacity for effective tissue repair \cite{agrawal2025role}. The precise control of spatiotemporal vasculature outgrowth is challenging and can lead to maladaptive remodeling \cite{santamaria2020remodeling}. During the sprouting of blood vessels, a specialized subset of endothelial cells known as tip cells is driven by the gradients of pro-angiogenic factors, such as vascular endothelial growth factor (VEGF), followed by stalk cells that extend the sprouting vessel for new vessel formation. Such processes are best modeled via partial differential equations (PDE), with chemotaxis.

Chemotaxis models of cell migration originate in the study of aggregation phenomena such as slime mold formation \cite{keller1970initiation}, and have since been adapted to angiogenesis, the growth of new vessels from an existing vasculature.
The setting that has received the most attention is tumor-induced angiogenesis. In the avascular phase a tumor is sustained by diffusion of
oxygen and nutrients alone; once this supply becomes limiting, hypoxic cells secrete angiogenic factors (most prominently VEGF) which diffuse through the extracellular matrix and reach the endothelial cells (ECs) of nearby venules. The ECs respond by
breaking down the basal lamina and advancing along the resulting chemical gradient, forming sprouts that subsequently interconnect into a capillary
network. This process, and its therapeutic interruption, have motivated an extensive mathematical literature \cite{othmer2004jmb,lankeit2023review}. A second setting calling for angiogenesis modeling is the ischemic heart disease, in which atherosclerotic obstruction of the coronary
vasculature limits myocardial perfusion and which remains a leading cause of mortality worldwide. Due to endogenous hypoxia (VEGF response is rarely
sufficient to restore perfusion), and the fact that many patients are unsuitable for revascularization, therapeutic angiogenesis \cite{simons2003therapeutic} (the exogenous stimulation of collateral vessel formation), has therefore been pursued as an alternative, which motivates quantitative descriptions of how endothelial cells translate chemical cues into vascular architecture.

Keller and Segel \cite{keller1970initiation,KELLER1971,keller1971traveling} were the first to investigate the mathematics behind slime mold aggregation, modeled via two coupled PDE's: 
\begin{equation}\label{ks_model}
  P_t = \nabla\cdot\bigl(D_P\nabla P - \chi\,P\,\nabla W\bigr),
  \qquad
  W_t = D_W\Delta W + g(P,W).
\end{equation} and derived the conditions for pattern-forming instability. The competing forces of the dissipative flux $D_{P}\nabla P$ and the chemotactic drift $-\chi P\nabla W$, was first identified by Keller and Segel, and has since been intensely investigated. If the chemotaxis is sufficiently strong, finite time blow-up of solutions is possible,  \cite{winkler2013finite}, particularly in higher dimensions ($n\geq3$) under certain sufficient conditions of initial data. Rascle and Ziti \cite{rascle1995finite} constructed symmetric self-similar solutions to show that finite-time blow-up is determined by the relative strengths of chemotactic attraction, chemical consumption, and cell diffusion.

As myxobacteria move under food deprived conditions, they not only aggregate but leave slime trails after moving, and the slime may not diffuse. That is, they change environments and tend to move towards a direction of prior cell movement. Inspired by myxobacteria movement, Stevens and Othmer \cite{othmer} derived a PDE system \cref{othmer_model}, from random walkers of particles that obey micro statistical rules. They proposed the following model, motivated via reinforced random walkers,
\begin{equation}\label{othmer_model}
  P_t = D\,\nabla\cdot\Bigl(P\,\nabla\ln\frac{P}{\Phi(W)}\Bigr),
  \qquad
  W_t = F(P,W),
\end{equation}
different from the Keller-Segel system \ref{ks_model}, albeit with competition between diffusion and chemotaxis.

Sleeman and Levine \cite{levine1997siam} further exploded the behavior of \cref{othmer_model} under a specific setting (see equation (4.1) in \cite{levine1997siam})
\begin{equation}\label{levine_model}
  P_t = D\Bigg(P_{xx}+a\Big(P\frac{W_x}{W}\Big)_x\Bigg),
  \qquad
  W_t = \lambda PW-\mu W,
\end{equation}
and found families of exact solutions \cite{levine2001theory,levine1997siam} of the blow-up and collapse cases based on certain parameters of $\Phi(W)$ and $F(P,W)$. Among their findings, they explained the formation of shock-form aggregation which depleted the cells in the surrounding region (this stable bounded aggregation structure is what one expects biologically). Levine, Sleeman, Nilsen-Hamilton and their coauthers further proposed several complex and detailed biochemical models with multi-species and multi-chemicals \cite{levine2000Nilsen_Hamilton,levine2001jmb,levine2009hamilton,levine2001bulletin}.

In recent years, a structurally different route to the same biological problem has attracted attention - the phase-field method. Extensive regularity and simulation results via this method have been obtained. A typical way is to encode angiogenesis assumptions into energy and entropy functionals of cells and chemicals, satisfying physical rules such as second law of thermodynamics. Agosti and Signori \cite{agosti2024analysis} introduced a multi-species Cahn--Hilliard--Keller--Segel system that couples a diffuse-interface description of the tumour phases to a Keller--Segel subsystem for chemotaxis and angiogenesis and establish weak
well-posedness for nonlinear potentials and some regularity and continuous dependence results. For tumor-induced angiogenesis, Vilanova, Colominas and Gomez \cite{vilanova2017mathematical} incorporated the key role played by filopodia during angiogenesis and the regression process of vascular systems through tumor angiogenic factors based on phase field method. Xu et al. employed phase field for developing a mesh-free method to study 3D angiogenesis with realistic model parameters obtained based on the information from image data \cite{xu2020phase}.

There is much motivation for non-linear or density dependent dispersal in the physical and life sciences \cite{stinga2023fractional, vazquez2006porous, antontsev2015evolution}. Among these are two popular models, the p-laplacian equation \cite{antontsev2015evolution} and the porous medium equation \cite{vazquez2006porous}. These models possess finite speed of propagation, as opposed to the infinite propagation speed of standard diffusion, making them realistic for several applications.
Various literature investigates p-Laplacian diffusion in place of Fickian diffusion. That is, the diffusive flux also depends on the gradient of cells (i.e. $|\nabla P|^{p-2}$). The case $1<p<2$ is of ``faster" diffusion and  ``slower" movement of the cells  is captured when $p>2$ - the general $p$-Laplacian permits both these dynamics \cite{antontsev2015evolution}.
This density-dependent flux resembles the volume-filling effect \cite{painter2002volume}, a standard way to prevent overcrowding (multiplying the motility and chemotatic sensitivity by a factor $q(P) = 1-P/P_{max}$ to switch the flux when approaching density saturation). The p-Laplacian diffusion naturally reproduces this mechanism: when $p>2$, $|\nabla P|^{p-2}\nabla P$) in the flux prevents an abrupt aggregation when $|\nabla P|$ is large. 
The effect of dispersal modeled via a $p$-Laplacian operator has been well investigated \cite{wang2023global, yang2022global, cong2016degenerate, rani2024global, li2020global, dewhirst2009dispersal}. The problem is degenerate and the analysis herein is more difficult than the case of linear diffusion \cite{jungel2010diffusive}. However, there is much motivation for non-linear dispersal in biology, \cite{chapman2015long, stinga2023fractional, song2019spatiotemporal, jungel2010diffusive}, as well as in physics and engineering sciences, as well \cite{bonforte2024cauchy, stinga2023fractional}. Two  popular methods to model non-linear dispersal (locally) are the porous medium equation \cite{vazquez2006porous} and the non-Newtonian filtration equation via the $p$-Laplacian operator \cite{antontsev2015evolution}. Herein, the ``type" or qualitative nature of the PDE changes or ``degenerates" in certain parts of the domain or certain parameteric regimes. These equations possess considerably greater mathematical challenges than the standard Laplacian counterpart - but can yield richer dynamics, such as finite time extinction (FTE) \cite{antontsev2015evolution}. Essentially, the porous medium equation degenerates when the solution $P \approx 0$, whilst the non-Newtonian filtration equation degenerates when $\nabla P \approx 0$ - yielding very different mathematical techniques for their individual analysis. 

The effect of dispersal via the $p$-Laplacian operator has recently been considered in chemotaxis systems, \cite{wang2023global}, as well as Lokta-Volterra (LV) type competition systems, \cite{yang2022global, rani2024global, upadhyay2026eco}. Kong and Liu  considered the Cauchy problem for the case of the $p$-Laplacian Keller-Segel model in spatial dimension $n\geq 3$, \cite{cong2016degenerate}, and under small data assumptions, prove the existence of global weak solutions. Certain decay estimates for the solution, as well as the extinction of the solution in finite time, have been established. Wang considered the chemotaxis-hapotaxis model with degradation of the chemical signal and $p$-Laplacian motion, \cite{wang2023global}. Herein, global existence of a weak solution is proved for $p> 1 + f(n)$, where $f$ is a general function that depends on the spatial dimension $n$, and $f \rightarrow 2$ as $n \rightarrow \infty$. Li considered an attraction-repulsion chemotaxis problem with  $p$-Laplacian, \cite{li2020global}, and proved global existence of weak solutions under certain parametric restrictions in the fast case ($1<p<2$), and for any parametric ranges in the slow case $(p>2)$. The case of slow diffusion for a system of two competitors both moving towards a chemical signal was first considered by Yang et al., \cite{yang2022global}, wherein a globally bounded weak solution was proved for any bounded initial data and any positive range of parameters. This was extended to the fast diffusion case in \cite{rani2024global}, where again globally bounded weak solutions were proved for the $1<p<\frac{3}{2}$ case, for any positive data, and for small positive data in the $\frac{3}{2}<p<2$ case. Recently, Zhuang et. al. considered the case of signal dependent sensitivity, with p-laplacian diffusion \cite{zhuang2021global}. The show global existence of weak solution for any initial data, for $p>p^{*}(n)$, where for $n=2$, $p^{*} = \frac{13}{6}$. Thus their result (in $n=2$) is not relevant to the case when $1<p<2$, or the ``fast" diffusion case.

In general, with chemotaxis-type mechanisms, the higher spatial dimension problem $(n\geq 3)$ becomes very challenging and finite time blow-up is possible, even with superlinear damping terms, \cite{winkler2018finite}. Thus, we expect only bounded weak solutions.

We consider p-Laplacian diffusion in \cref{levine_model} and posit our basic model,
\begin{equation}\label{base_model_0}
\begin{cases}
\displaystyle
P_t = D\,\frac{\partial}{\partial x}\!
      \Bigl[\bigl(|P_x|^{p-2}+\theta\bigr)P_x
            - P \frac{\partial}{\partial x}ln\Phi(W) \Bigr],
\\[5pt]
\displaystyle
W_t = \Bigl(\frac{P}{1+\nu W}-\mu\Bigr)W.
\end{cases}
\end{equation}
We find when $1<p<2$, the cell and chemical kinetics evolve cell density towards aggregation, that can cause cells drive aggregation areas to extinction. The aggregation is observed to be bimodal or multi modal, a dynamic different from  existing models, to the best of our knowledge. When $p>2$, the introduction of $|\nabla P|^{p-2}$ models a  ``slower" diffusion process. Biologically, we can consider this behavior as a ``frozen" effect of cell density, as $P(x,t)$ maintains this behavior for long time and can overcome diffusive dampening (See simulations in \cref{simulations}).
 We also add a logistic term to $P$, and regular diffusion to $W$, to obtain \cref{eq:logistic} in \cref{model_sim}. Energy estimates and analysis are performed on \cref{eq:logistic} in \cref{energy}. We prove that there exists globally bounded weak solutions for  \cref{eq:logistic} for $3/2<p<2$ and $\theta=0$. We also perform linear analysis near steady states of \cref{eq:logistic} to obtain instability conditions of pattern-forming in \cref{pattern}. We find patterns have enhanced connected forms with quick saturation when $p>2$, and patterns start to emerge only when noise amplitude exceeds a critical value, leading to very similar bimodal forms.

The manuscript is structured as follows. \cref{intro} gives the biological background and motivation of this work; \cref{model_sim} describes the investigated equations and gives the simulation results both in 1D and 2D, which drives our investigation; \cref{energy} gives our recent regularity results; \cref{pattern} explores pattern-forming behavior  based on instability conditions from linear analysis; \cref{discussions} gives our  conclusions, discussions and describes future work.

\section{Model descriptions and Simulation}\label{model_sim}
\subsection{Model system}
The main model we present is aimed to further investigate aggregation behavior, based on Stevens and Othmer \cite{othmer}, and Sleeman and Levine \cite{levine1997siam}. The shock-form aggregation they found, which depletes mass near the aggregation area,  drives part of our current research. Sleeman and Levine \cite{levine1997siam} derived a family of classical solutions to designate between the finite-time blow up behavior and the collapsing of solution behavior. They analyzed the shock-form aggregation as a transformation from blow up case to collapse case, similar to analysis in hyperbolic PDE. Biologically speaking, the shock-form aggregation is  meaningful, as it leads to strong and stable aggregation. Besides the competition between chemotaxis effect and (Fickian) diffusion dampening, we incorporate p-Laplacian diffusion and find that p-Laplacian diffusion leads to very different dynamics than any other reported, to our best of knowledge - in particular the shock form that occurs with regular diffusion, can be dampened with slow diffusion. Here are governing equations posed on $\Omega \subset \mathbb{R}^n$
\begin{equation}\label{eq:base_model}
\begin{cases}
P_t = D\,\nabla\cdot\Big[(|\nabla P|^{p-2}+\theta)\nabla P - P\,\dfrac{\delta}{(W+\beta)(W+\gamma)}\,\nabla W\Big],\\[6pt]
W_t = \Big(\dfrac{P}{1+\nu W}-\mu\Big)W,
\end{cases}
\end{equation}
where $\nu=1/\lambda$, $\delta/[(W+\beta)(W+\gamma)] = \Phi'(W)/\Phi(W)$
with $\Phi(W)=[(W+\beta)/(W+\gamma)]^{a}$, and all parameters are positive.
And apply zero-flux (homogeneous Neumann) boundary conditions on $P$ upon $\partial\Omega$:
\begin{equation}
\Big[(|\nabla P|^{p-2}+\theta)\nabla P - P\,\frac{\delta}{(W+\beta)(W+\gamma)}\,\nabla W\Big]\cdot\nu = 0.
\label{base_boundary}
\end{equation}
Because $W$ satisfies a pointwise ODE (no spatial derivatives), no boundary condition is required for $W$.
In addition, we introduce a density-Dependent survival modulation on \cref{eq:base_model} (Take one spatial dimension here upon $[0,L]$)
\begin{equation}\label{eq:logistic}
\begin{cases}
\displaystyle
P_t = D\,\frac{\partial}{\partial x}\!
      \Bigl[\bigl(|P_x|^{p-2}+\theta\bigr)P_x
            - P\,\frac{\delta}{(W+\beta)(W+\gamma)}\,W_x\Bigr]
      + rP\!\left(1-\frac{P}{K}\right),
\\[10pt]
\displaystyle
W_t = D_{1}W_{xx} + \Bigl(\frac{P}{1+\nu W}-\mu\Bigr)W,
\end{cases}
\end{equation}
where $r>0$ is the survival--apoptosis modulation rate and $K>0$ is
the density threshold at which survival and apoptosis balance. And zero flux boundary conditions are applied on $P$ (i.e. \cref{base_boundary}) and $W$, where $[W_x]_{x=0,\,L} = 0$.
Compared with the base model, a logistic source term
$rP(1-P/K)$ is added to the $P$-equation,
which breaks mass conservation: the total cell mass
$\int_0^L P\,dx$ now satisfies
\[
  \frac{d}{dt}\int_0^L P\,dx
  = \underbrace{\bigl[\text{flux}\bigr]_0^L}_{=\,0}
    + \int_0^L rP\!\Bigl(1-\frac{P}{K}\Bigr)dx.
\]
Also, Fickian diffusion $D_1 W_{xx}$ is incorporated, as there is much biological motivation for diffusion in the literature for VEGF.

\subsection{Cellular Aggregation with p-Laplacian Diffusion}\label{simulations}
In this section, we show 1D and 2D simulations of model \eqref{eq:base_model} and the model with logistic control \eqref{eq:logistic}. Rich dynamics, such as bimodal aggregation ($1<p<2$) and triangular aggregation ($p>2$), are found in these models, which motivates our investigation. \cref{fig:base_p2} ($p=2$) reproduces the shock form from Levine and Sleeman \cite{levine1997siam}. \cref{fig:base_p3} ($p>2$) shows that the aggregation with slow diffusion does not lead to strong celluar aggregation (that is the aggregation is fairly slow). \cref{fig:base_p15} ($p<2$) shows the bimodal form that preserves the property that locally there is depletion of cells. Before reaching the final profile, we see a local minimum emerges between the two peaks but then disappears due to a mechanism that warrants further investigation. We conjecture this bimodal behavior (or even through multi-modal) is strongly related to initial conditions and the form of the source term in the equation for the chemical $W$. It is also observed that adjusting $p$ has a strong impact on the behavior of our base model \cref{eq:base_model}: increasing $p$ changes the cellular aggregation from bimodal to shock-form, then to hump-form (shown in \cref{fig:base_p_scan}). \cref{fig:base_theta_scan} shows that varying normal diffusion $\theta$ also modulates the dynamics for $1<p<2$. As for our model \cref{eq:logistic} with logistic control of $P$, that is in the cell, \cref{fig:logistic_K_scan} 
illustrates that this logistic term can activate the celluar proliferation at positions of almost extinct cells (i.e. near boundaries when $p<2$ in \cref{fig:logistic_K_scan}). However, a strong logistic term with large carrying capacity $K$ can dampen the cell profiles to a nearly constant steady state of $P=K$, especially for $p\geq2$. We also numerically do experiments on the 2D compartments of our 1D simulations. And 2D simulations essentially align with our 1D tests. In addition, we perform experiments on a 2D cylinder to simulate angiogenesis on a capillary.

\begin{figure}[htbp]
  \centering
  \includegraphics[width=0.92\textwidth]{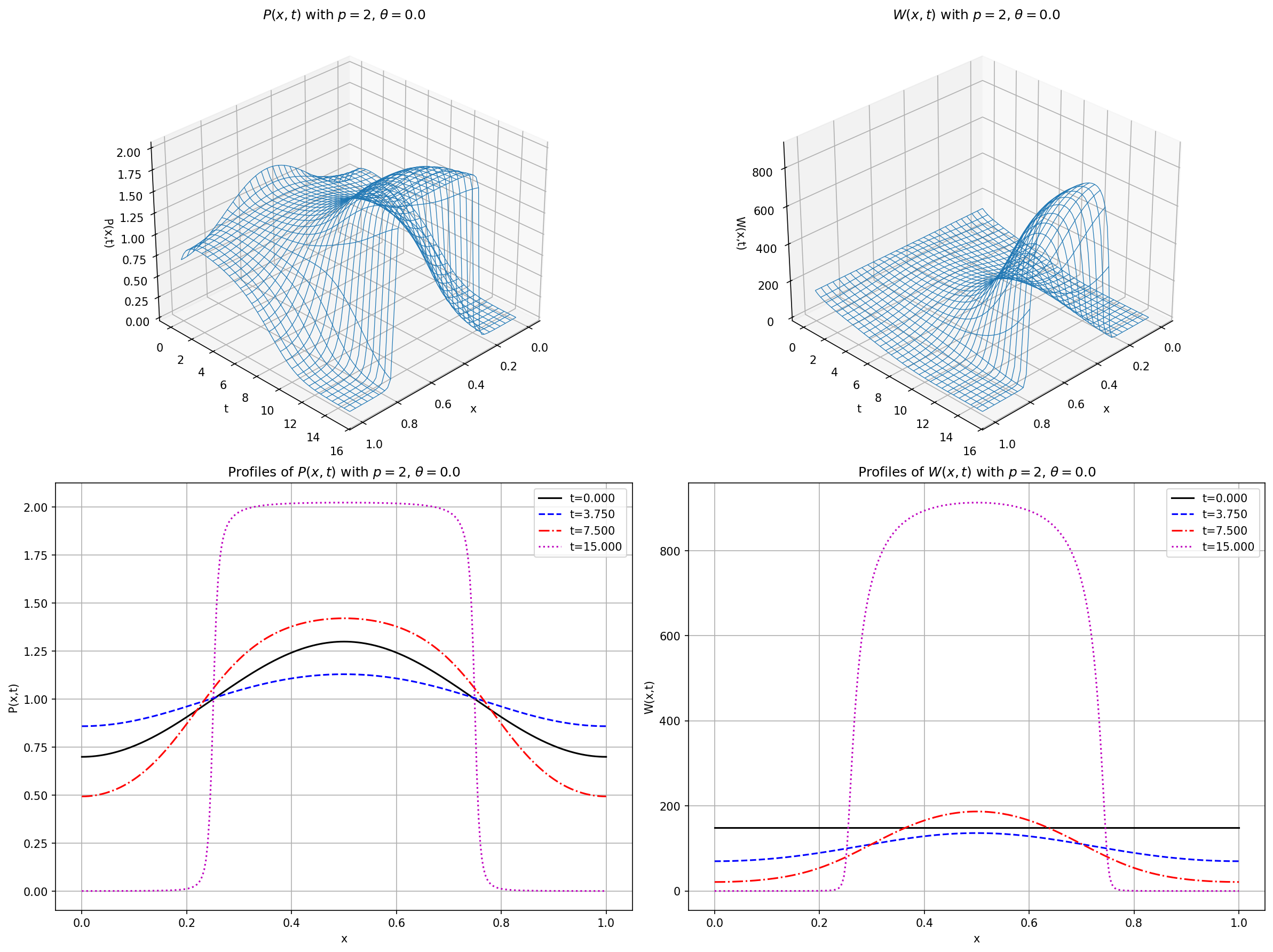}
  \caption{Base model with $p=2$, $D = 0.036$, $\mu = 1.0$, $\lambda = 10^{3}$, $\beta = 0.01$, $\gamma = 10^{2}$, $\delta = 10^{2}$, $L = 1.0$, $\theta=0$, $P(x,0)=1-0.3\cos(2\pi x)$, $W(x,0)=e^5$.
           \textit{Top row}: 3D surface plots of $P(x,t)$ (left) and $W(x,t)$ (right)
           over $x\in[0,1]$, $t\in[0,15]$.
           \textit{Bottom row}: spatial profiles at four representative times
           $t=0,\,3.75,\,7.5,\,15$.
           $P$ develops a sharp shock-like aggregation front that concentrates
           mass in the interior of $[0,1]$; $W$ forms a corresponding high-VEGF
           plateau co-localised with the cell aggregate.}
  \label{fig:base_p2}
\end{figure}

\begin{figure}[htbp]
  \centering
  \includegraphics[width=0.92\textwidth]{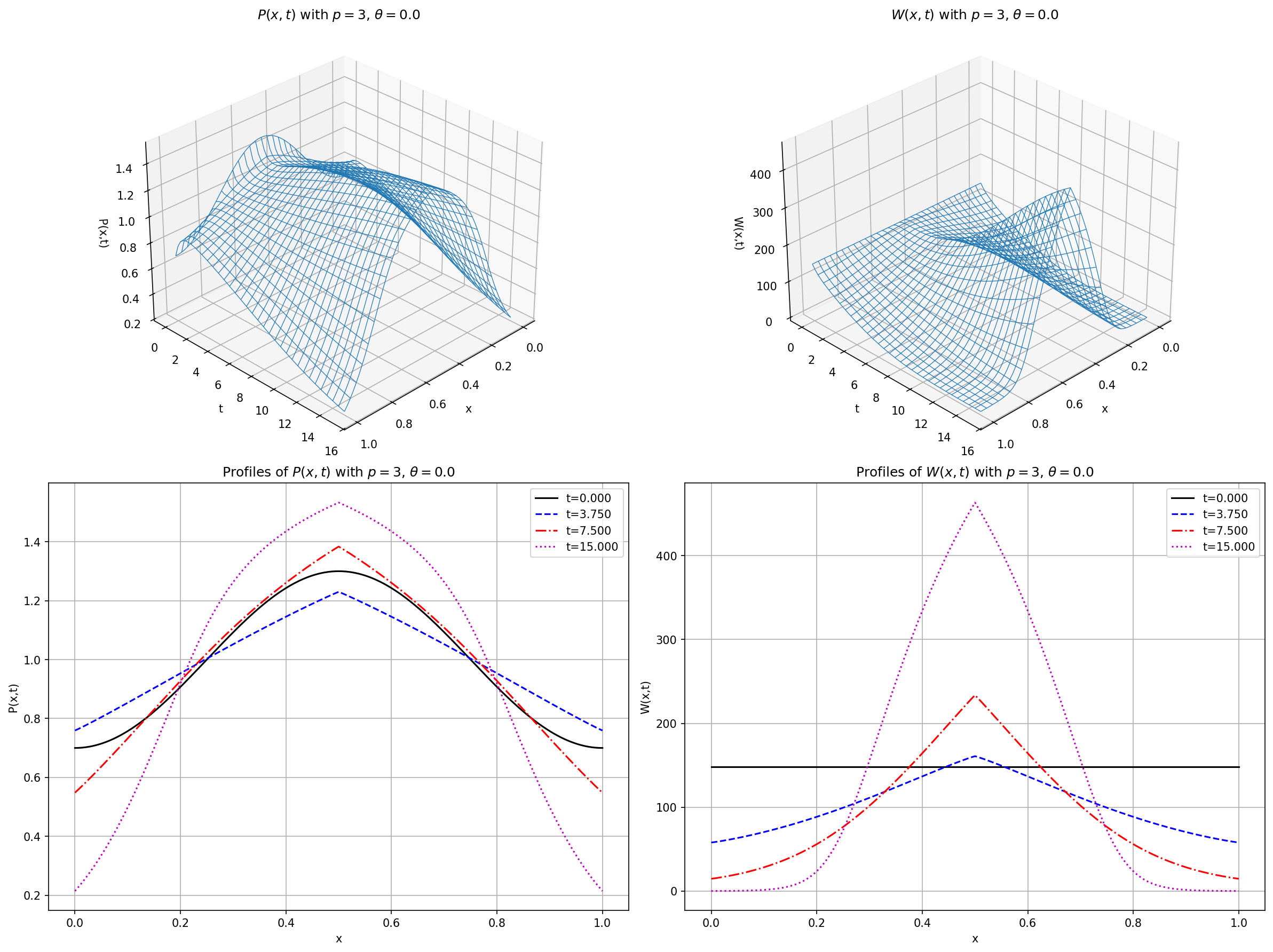}
  \caption{Base model with $p=3$ and other parameters consistent with \cref{fig:base_p2}.
         The stronger nonlinear diffusion ($p=3$) produces a single
         smooth aggregation peak in $P$ that is narrower and more
         concentrated than the $p=2$ shock front
         (Figure~\ref{fig:base_p2}), with the peak shifting towards
         $x\approx 0.5$ by $t=15$.
         Unlike the sharp step-like shock at $p=2$, the profile here
         decays smoothly to near zero at the boundaries.
         $W$ develops a single central maximum ($W\sim 450$ at $t=15$),
         substantially lower than the bimodal VEGF peaks seen at $p=1.5$
         (Figure~\ref{fig:base_p15}).}
  \label{fig:base_p3}
\end{figure}

\begin{figure}[htbp]
  \centering
  \includegraphics[width=0.92\textwidth]{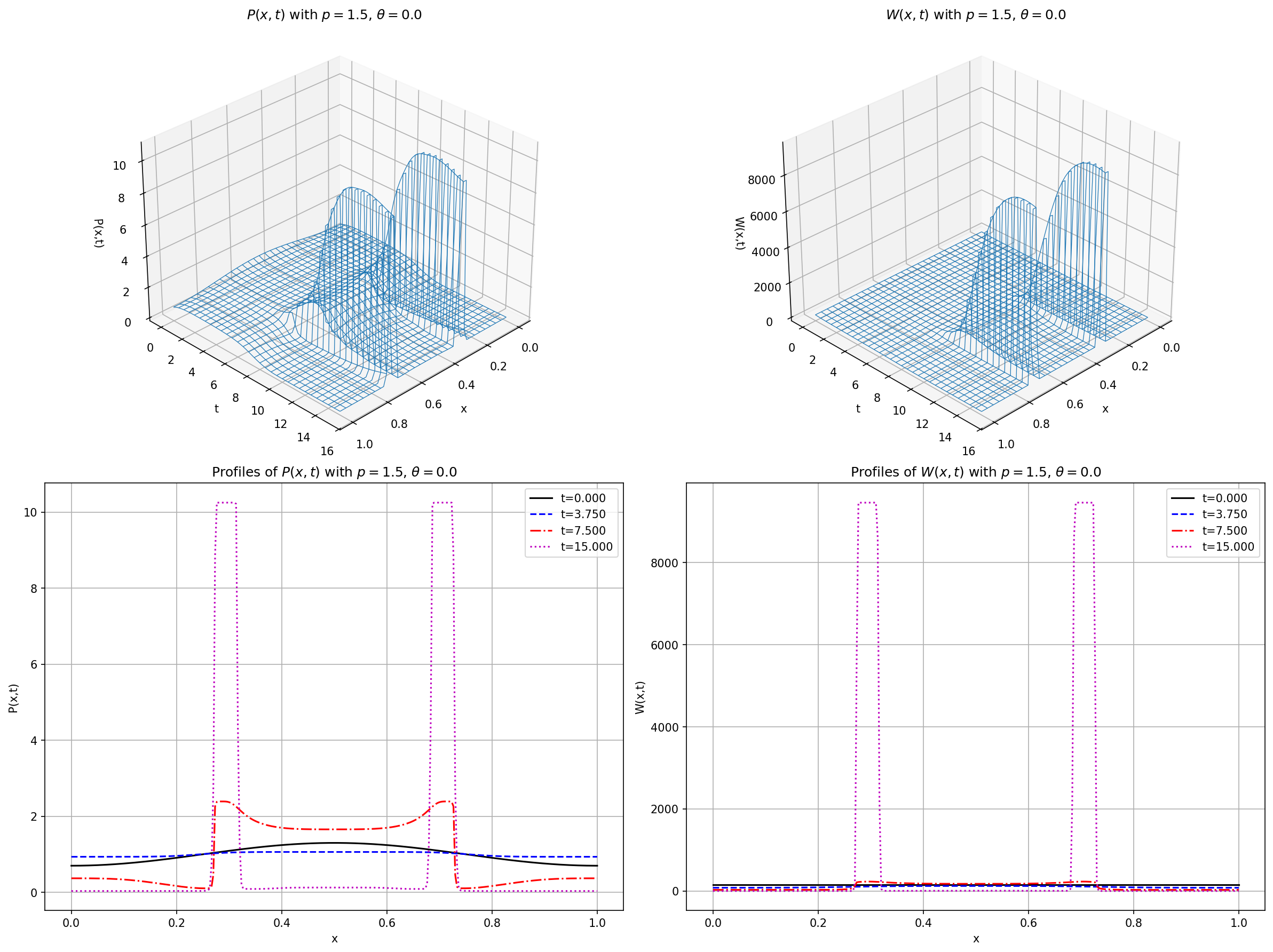}
  \caption{Base model with $p=1.5$ and other parameters consistent with \cref{fig:base_p2}.
           Sub-quadratic diffusion ($p<2$) gives rise to a
           \emph{bimodal} aggregation pattern: $P$ develops two sharp
           symmetric peaks near $x\approx 0.3$ and $x\approx 0.7$,
           with a depleted trough in between.
           The corresponding $W$ profile mirrors this structure,
           exhibiting two co-localised VEGF maxima of very large amplitude.
           This multi-peak behaviour is absent for $p\ge 2$ and is a
           distinctive feature of fast ($p<2$) $p$-Laplacian diffusion.}
  \label{fig:base_p15}
\end{figure}

\begin{figure*}[!t]
  \centering
  \includegraphics[width=\textwidth]{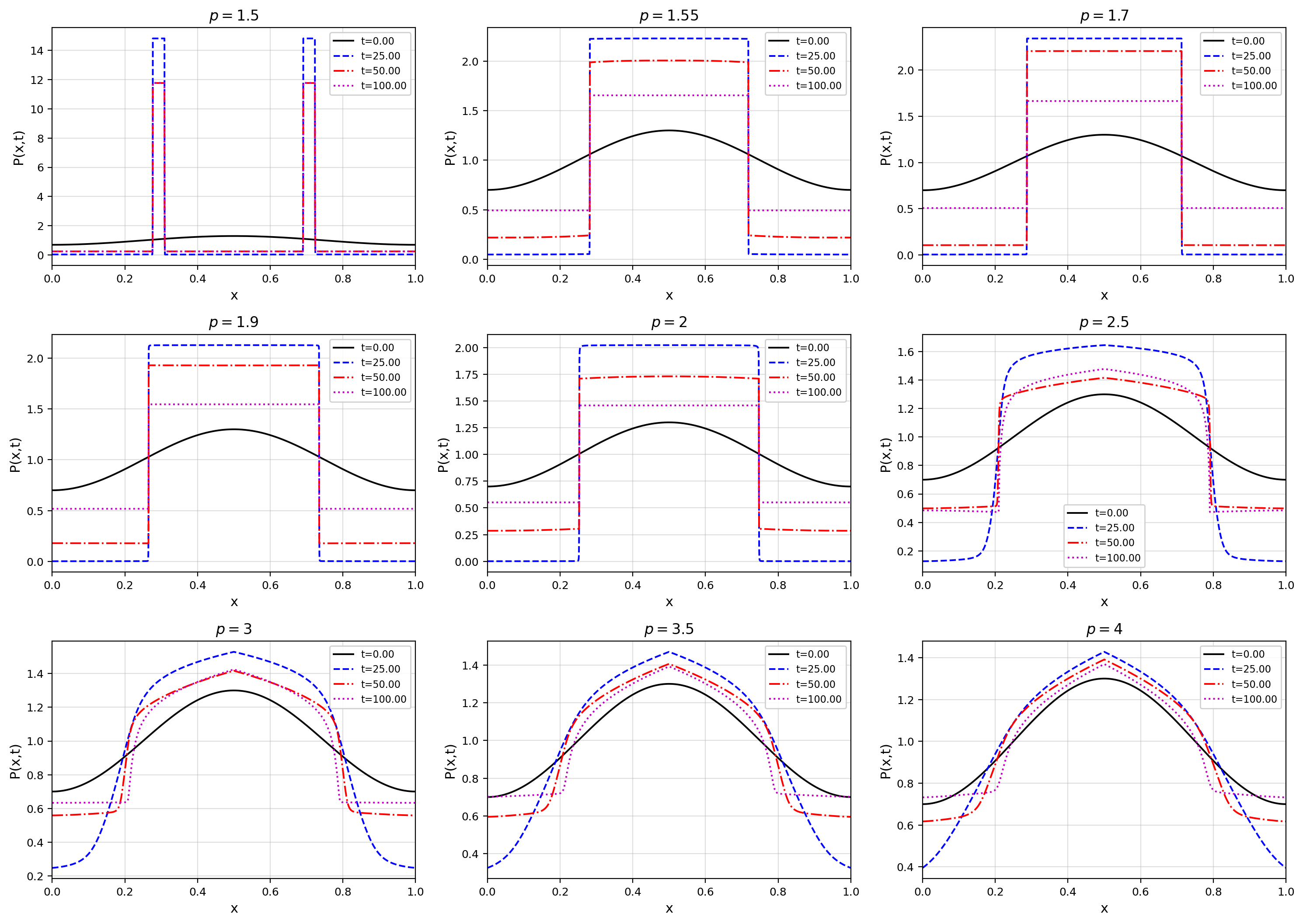}
  \caption{%
    Profiles of the cell density $P(x,t)$ at
    $t=0,\,25,\,50,\,100$, for nine values of the diffusion exponent $p$
    increasing from $p=1.5$ (top left) to $p=4$ (bottom right) and other parameters consistent with \cref{fig:base_p2}.
    The exponent controls the regularity of the aggregate.
    For $p\ge 3$ the initial cosine is amplified and mildly steepened but
    remains smooth for all times; at $p=2.5$ the profile flattens into a
    broad hump with steep shoulders.
    For $1.55\le p\le 2$ the solution collapses onto a plateau: the plateau height decays with time while the outer
    region refills, indicating slow relaxation towards the spatially
    homogeneous state for long time run.
    At $p=1.5$ this plateau is replaced by two narrow spikes pinned at the
    interface positions, attaining $P\approx 15$ at $t=25$ with near-total
    depletion elsewhere (even for long time run).
  }
  \label{fig:base_p_scan}
\end{figure*}

\begin{figure}[htbp]
  \centering
  \includegraphics[width=0.92\textwidth]{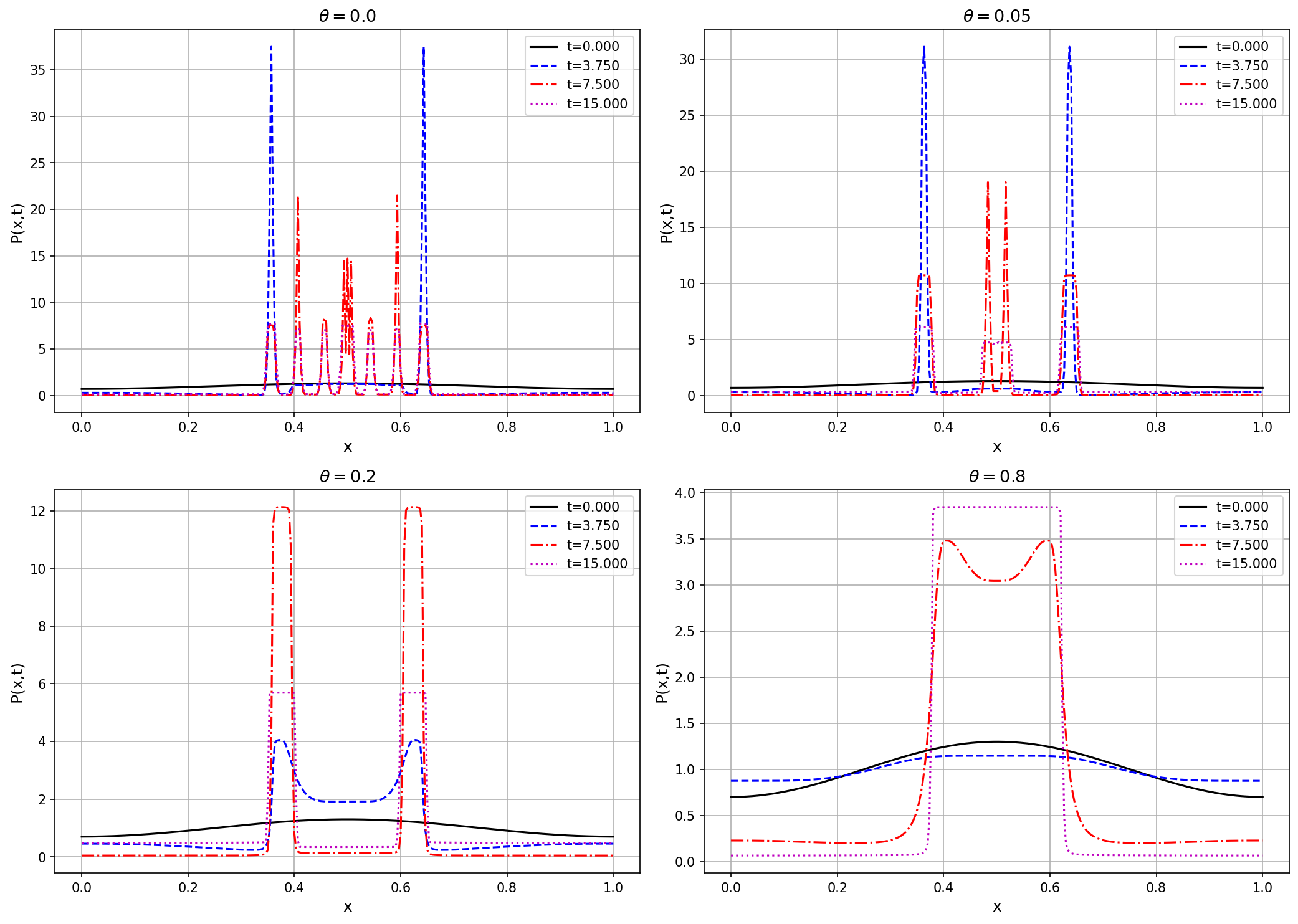}
  \caption{Profiles of the cell density $P(x, t)$: effect of increasing linear diffusion $\theta$ on the
           aggregation pattern for $p=1.5$, $W(x,0)=e^0=1$, and unmentioned parameters consistent with \cref{fig:base_p2}.
           Each panel shows $P(x,t)$ profiles at $t=0,\,3.75,\,7.5,\,15$.
           \textit{Top-left} ($\theta=0$): strongly multi-modal peaks with
           very high amplitude.
           \textit{Top-right} ($\theta=0.05$): peak heights are reduced but
           the bimodal structure persists.
           \textit{Bottom-left} ($\theta=0.2$): further damping; the two
           peaks narrow and their amplitude decreases markedly.
           \textit{Bottom-right} ($\theta=0.8$): a single broad aggregation
           peak replaces the two-peak structure, indicating that normal diffusion has a strong damping effect.}
  \label{fig:base_theta_scan}
\end{figure}

\begin{figure}[htbp]
  \centering
  \includegraphics[width=0.95\textwidth]{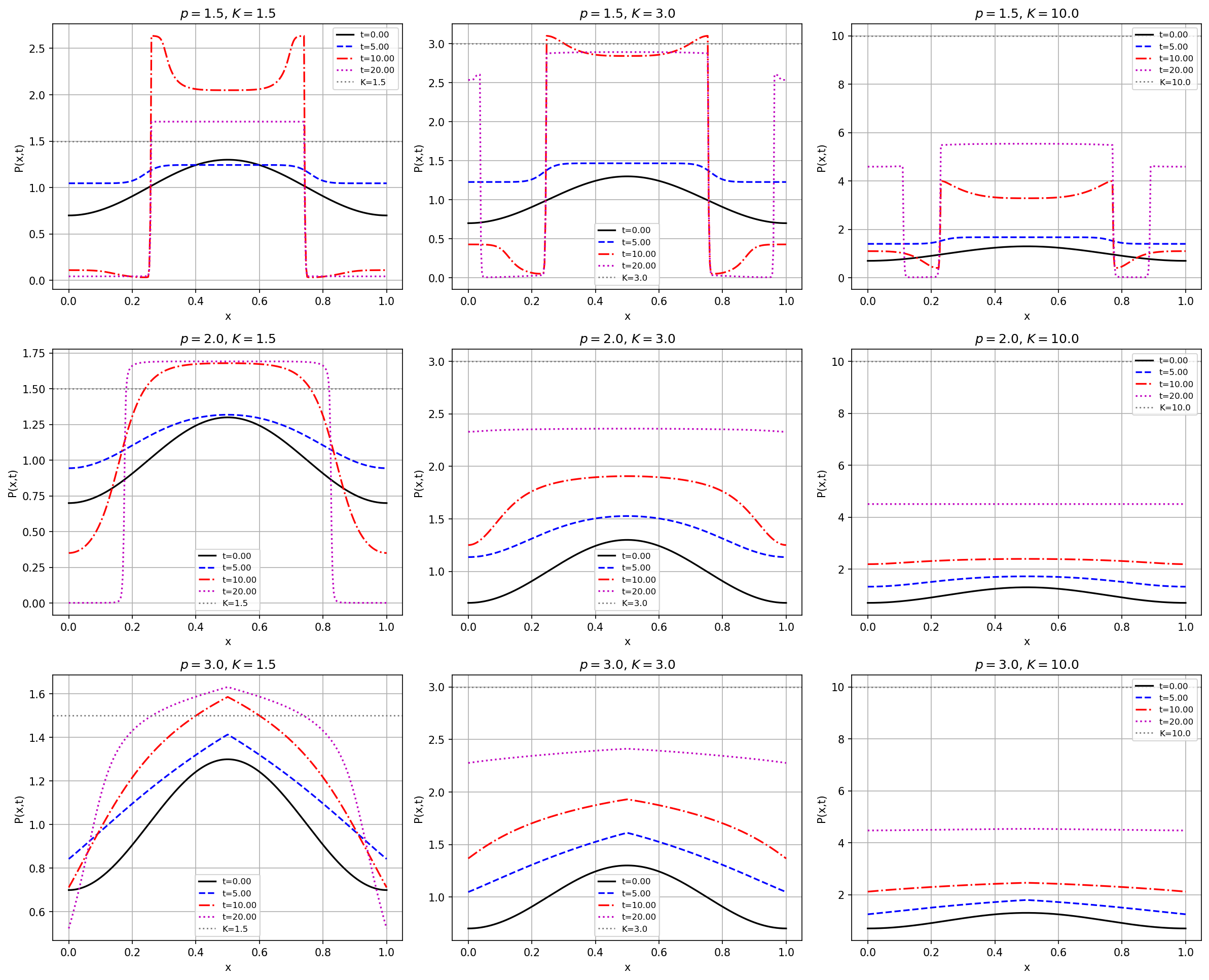}
  \caption{Profiles of $P$ of logistic survival modulation \cref{eq:logistic}:
           effect of the carrying capacity $K$, with $r=0.1$,
           $W_0=e^5$, $T=20$ and unmentioned parameters consistent with \cref{fig:base_p2}.
           Rows correspond to $p=1.5,\,2,\,3$; columns to
           $K=1.5,\,3,\,10$.
           For $p=1.5$ (top row) a small cap ($K=1.5$) suppresses the
           bimodal spikes to moderate amplitude; at $K=3$ two sharp
           peaks near the \textbf{boundaries} re-emerge by $t=10$ and then
           partially collapse; at $K=10$ the cap is more inactive and the
           peaks grow more freely, which also leads to aggregations near boundaries.
           For $p=2$ (middle row) all three panels show smooth profiles that grow toward $K$; the solutions gradually tend to be space-independent for $K=3,\ 10$.
           For $p=3$ (bottom row) profiles are unimodal and
           with triangular-form across all $K$ values, with different amplitude tracking
           $K$; the profiles also gradually tend to be space-independent toward $K$ for $K=3,\ 10$.}
  \label{fig:logistic_K_scan}
\end{figure}

The 2D solution is obtained on both Cartesian plane, and cylindrical artery-like geometry. \cref{fig:2d_b_p} shows that the aggregation dynamics in the 2D simulations exhibit similar behavior to the 1D simulations for normal diffusion ($p=2$) and slow diffusion ($p>2$). However, they diverge significantly for fast diffusion ($1<p<2$). For example, the sharp bimodal aggregation observed in 1D at $t=15$ is instead trapped into attenuated discrete clusters in 2D over the same time frame. This stems from the transverse gradient coupling of the $p$-Laplacian in 2D, which is absent in 1D. In 2D, the non-linear diffusion coefficient scales with $(P_x^2+P_y^2)^{(p-2)/2}$. Because $p<2$, this exponent is negative, and a steep gradient in the transverse $y$-direction diminishes the diffusion coefficient in the $x$-direction. Therefore, a steep gradient in any single direction forces the non-linear contribution of the diffusivity toward zero. This restricts mass transport in all directions and localizes the clusters. In contrast, the exponent $(p-2)/2$ vanishes for $p=2$, or remains positive for $p>2$. In these cases, the transverse gradient is neutralized, or is preserved to enhance local diffusion and mass transport. This results in more striking aggregation fronts depicted in \cref{fig:2d_b_p} for normal and fast diffusion, $p=2.0$ and $p=3.0$, respectively, by preventing mass from accumulating in local clusters.

The comparison of \cref{fig:2d_b_p} with \cref{fig:base_p2} reveals a decreased aggregation in higher spatial dimensions, where the 2D profile at $t=15$ resembles the 1D profile at $t=7.5$. The delay occurs because distributing the mass over a planar area yields lower initial peak concentrations than 1D compression. This reduced local density slows the localized evolution of the chemical field, which weakens the chemotactic drift.

\begin{figure}[!htbp]
\centering
\includegraphics[width=0.95\textwidth]{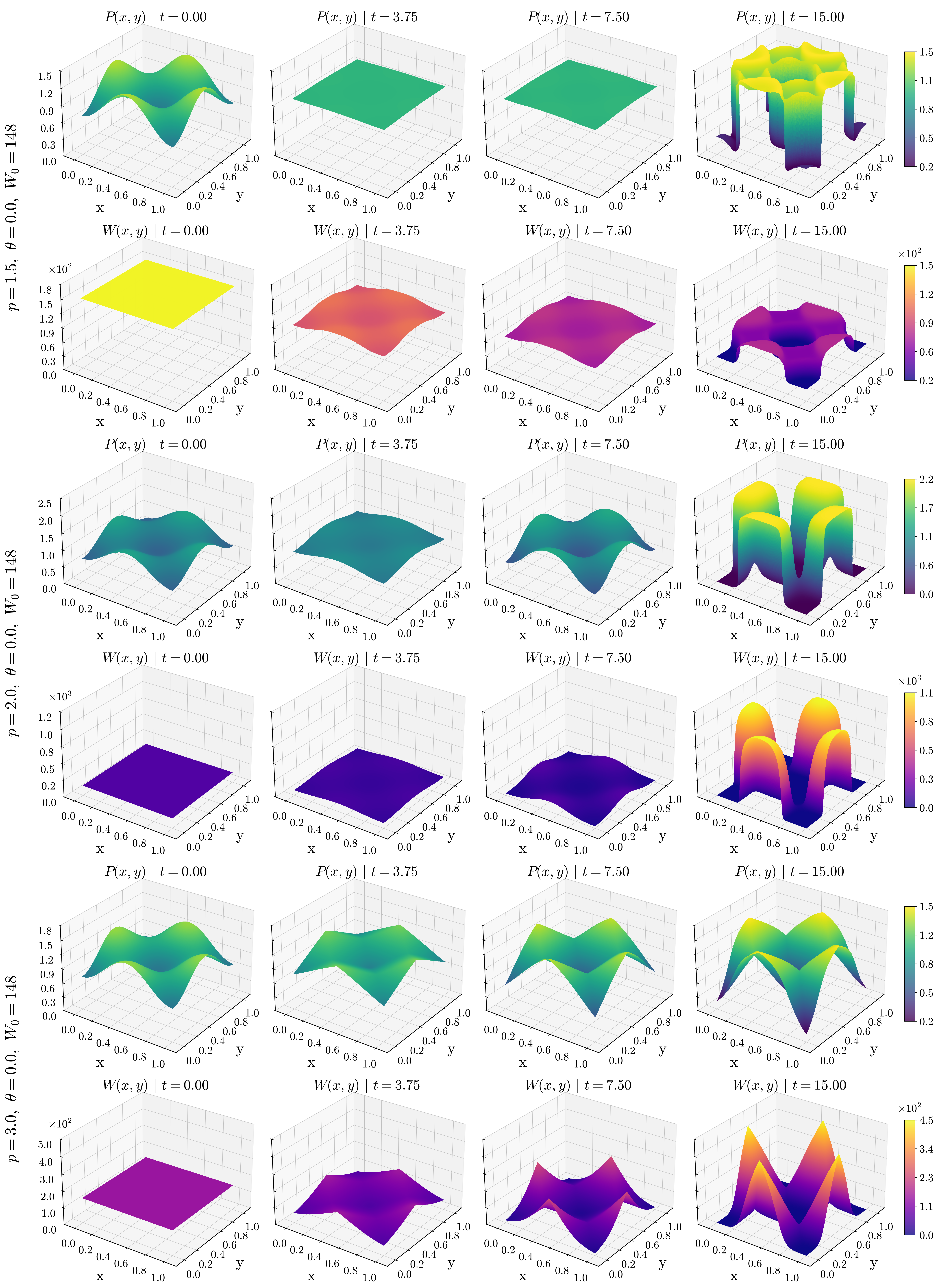}
\caption{Cell aggregation and VEGF concentration, $P(x,y,t)$ (top rows) and $W(x,y,t)$ (bottom rows), respectively, on a planar domain, organized into three blocks based on the diffusion regime. Base model is employed with $D = 0.036$, $\mu = 1.0$, $\lambda = 10^{3}$, $\beta = 0.01$, $\gamma = 10^{2}$, $\delta = 10^{2}$, $\theta=0$, $L = 1.0$, $P(x,y,0)=1.0-0.3\cos(2\pi x)\cos(2\pi y)$, $W(x,y,0)=e^5$ over $[0,1]\times[0,1]$, shown at four representative times of $t=0,\,3.75,\,7.5,\,15$. \textit{Upper block}: fast diffusion ($p=1.5$). \textit{Middle block}: normal diffusion ($p=2.0$). \textit{Lower block}: slow diffusion ($p=3.0$). Rather than forming continuous shock-like fronts, $P$ develops into symmetric clusters across the domain. In fast diffusion, transverse gradient localizes the mass into attenuated aggregates, with $W$ forming plateaus that co-localized with the cell aggregates.}
\label{fig:2d_b_p}
\end{figure}

Cylindrical topology shown in \cref{fig:2d_b_c} introduces periodic boundary conditions along the azimuthal axis ($\phi$) which significantly alters the aggregation dynamics from a planar model. The clusters evolved into interconnected structures on the cylinder. This mimics the initial budding and localized aggregation in vascular sprouting \cite{dudley2023pathological}. The azimuthal gradients ($\nabla \phi$) diminish the axial diffusivity, where the continuous bands of $P$ and $W$ fields are forced to pinch.

\begin{figure}[!htbp]
\centering
\includegraphics[width=0.95\textwidth]{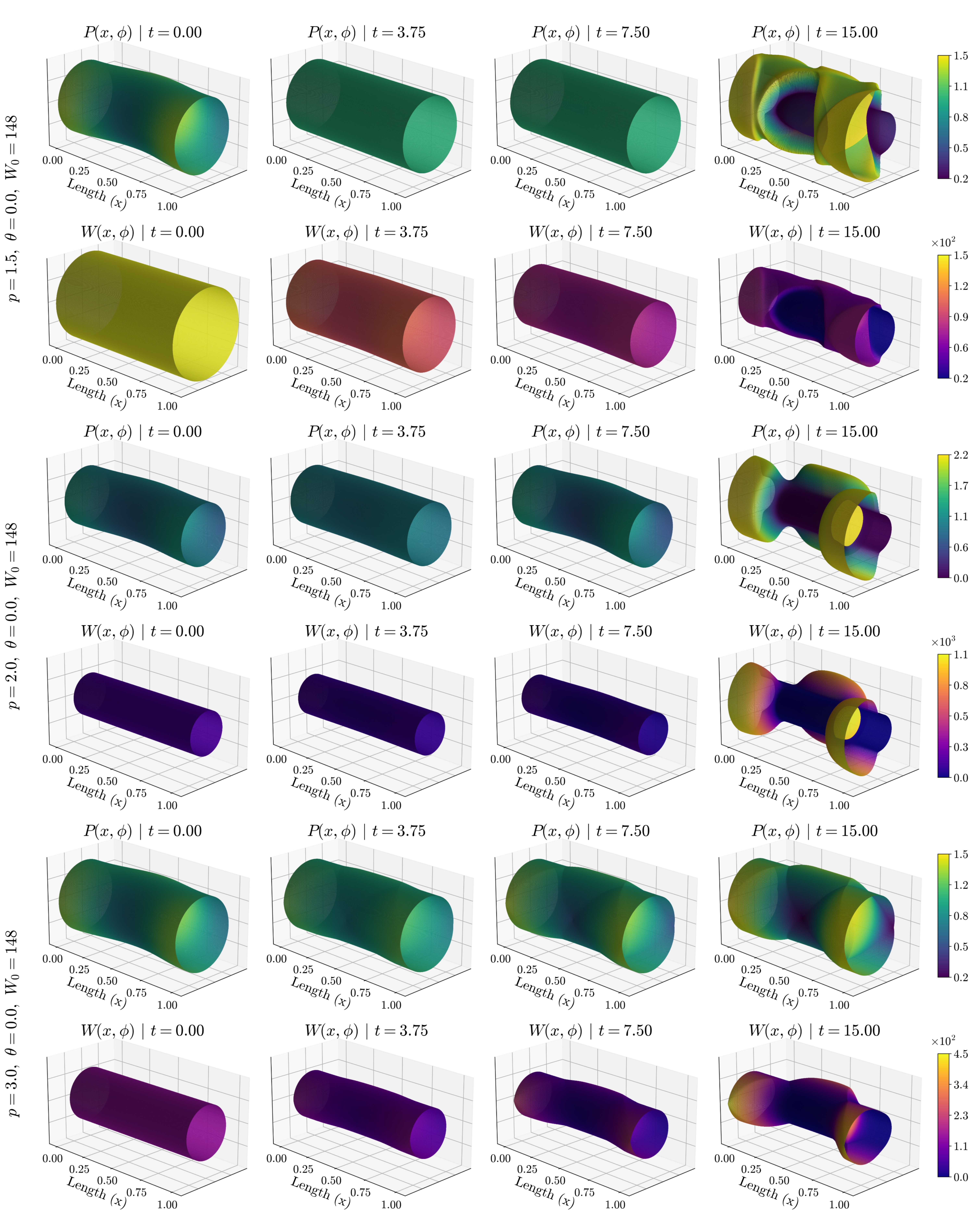}
\caption{Cell aggregation and VEGF concentration, $P(x,\phi,t)$ (top rows) and $W(x,\phi,t)$ (bottom rows), respectively, on a cylindrical domain representing a real blood vessel. The figure is organized into three blocks based on the diffusion regimes, fast ($p=1.5$, upper block), normal ($p=2.0$, middle block), and slow ($p=3.0$, lower block). Base model is employed and other parameters are consistent with \cref{fig:2d_b_c}. The concentrations are visualized as radial deformations from a baseline radius $R_0$, with initial conditions $P(x,\phi,0)=1.0-0.3\cos(2\pi x)\cos(\phi)$ and $W(x,\phi,0)=e^5$ over the cylindrical surface defined by axial length $x \in [0,1]$ and azimuthal angle $\phi \in [0,2\pi)$, shown at four representative times ($t=0,\,3.75,\,7.5,\,15$). Sprouting-like buds emerge along the vessel surface. Consistent with the planar geometry, $W$ develops VEGF plateaus that co-localized with the cell aggregates.}
\label{fig:2d_b_c}
\end{figure}

The volume-filling factor $q(P)=max(1-P/P_{max},0)$ is employed to modify the chemotaxis term in the base model to prevent infinite densities. \cref{fig:2d_vf_p} shows the coupling effect of $q(P)$ on cell aggregation and VEGF distributions on a planar domain, whereas \cref{fig:2d_vf_c} associates with the cylindrical domain. As the local cell density approaches the maximum admissible value, $P \rightarrow P_{max}$, the volume-filling factor approaches zero. And the chemo-tactic flux vanishes while random diffusion persists. This prevents blow-up at the flux level. For a much lower $W(x,y,0)$ than the previous cases, \cref{fig:2d_vf_p} and \cref{fig:2d_vf_c} show several sharp aggregation peaks. The impact of the $q(P)$ is most clearly seen in the fast diffusion regime ($p=1.5$, upper blocks), where several neighboring peaks in the form of clusters are plateaued as the cell density hits $P_{max}$. In contrast, the normal and slow diffusion regimes ($p \ge 2$) concentrate the mass into multi-modal peaks that have not yet reached the $q(P)$ threshold within the same time-frame.

\begin{figure}[!htbp]
\centering
\includegraphics[width=0.95\textwidth]{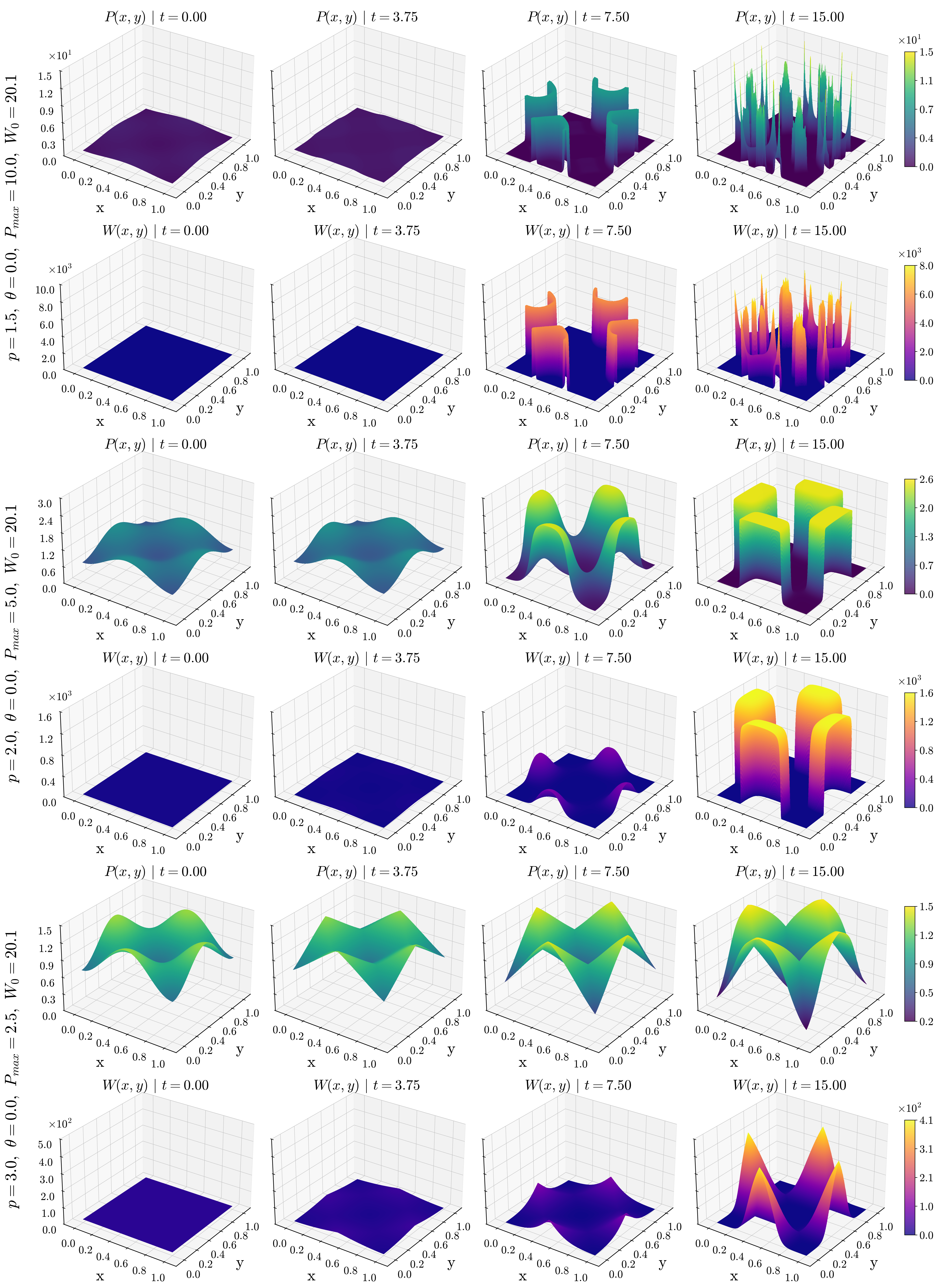}
\caption{Cell aggregation and VEGF concentration, $P(x,y,t)$ (top rows) and $W(x,y,t)$ (bottom rows), respectively, on a planar domain, organized into three blocks based on the diffusion regime. Base model with volume-filling factor is employed with $D = 0.036$, $\mu = 1.0$, $\lambda = 10^{3}$, $\beta = 0.01$, $\gamma = 10^{2}$, $\delta = 10^{2}$, $\theta=0$, $L = 1.0$, $P(x,y,0)=1.0-0.3\cos(2\pi x)\cos(2\pi y)$, $W(x,y,0)=e^3$ over $[0,1]\times[0,1]$, shown at four representative times of $t=0,\,3.75,\,7.5,\,15$. \textit{Upper block}: fast diffusion ($p=1.5$). \textit{Middle block}: normal diffusion ($p=2.0$). \textit{Lower block}: slow diffusion ($p=3.0$).}
\label{fig:2d_vf_p}
\end{figure}

The localized clusters for $p=1.5$ are more pronounced for the cylindrical domain shown in \cref{fig:2d_vf_c}, where thickened geometric patches are formed rather than singular peaks and rings. Once an aggregation band reaches $P_{max}$, the persisting random motility forces the cells to distribute along the azimuthal ($\phi$) and axial ($x$) directions, yielding to localized swellings.

\begin{figure}[!htbp]
\centering
\includegraphics[width=0.95\textwidth]{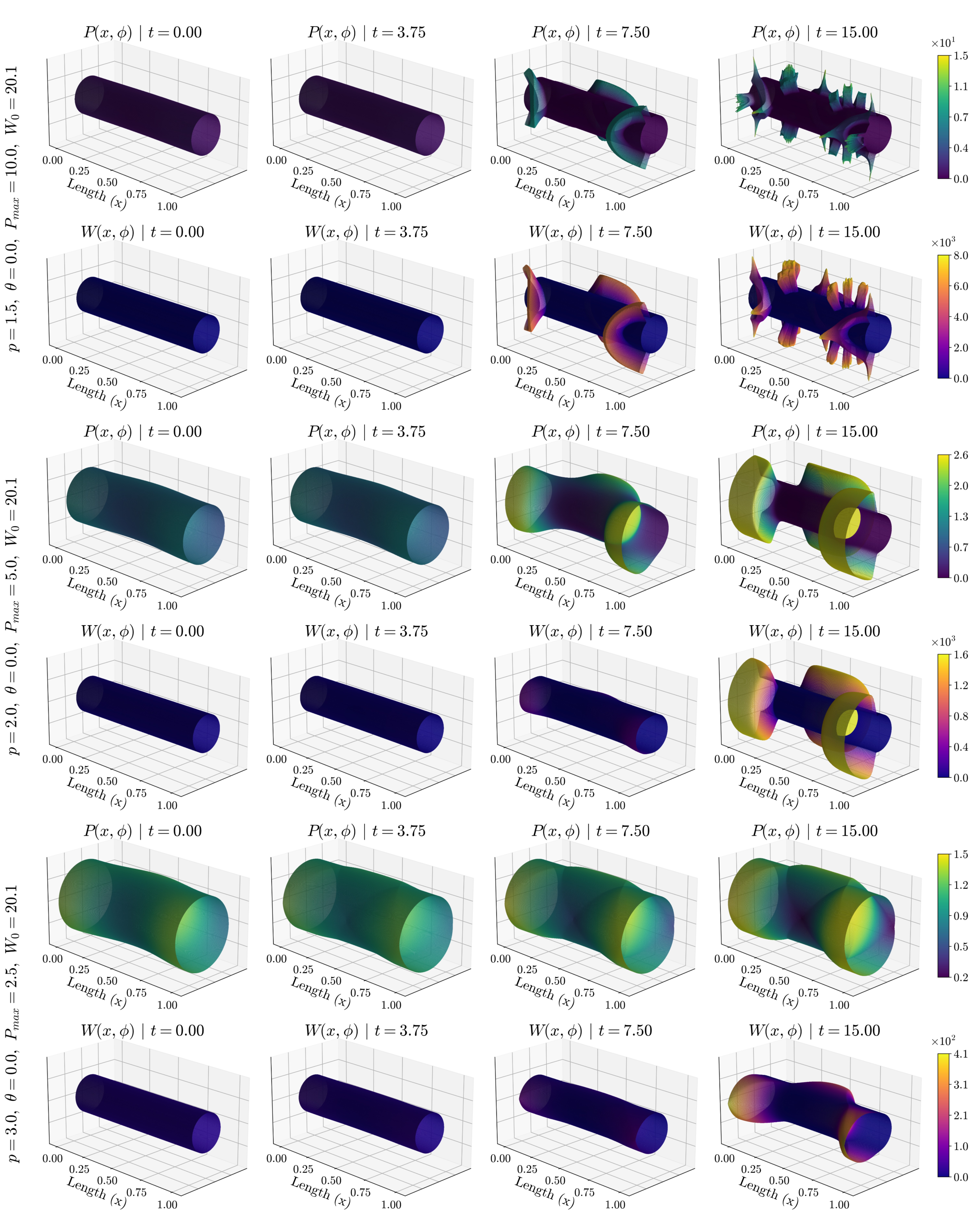}
\caption{Cell aggregation and VEGF concentration, $P(x,\phi,t)$ (top rows) and $W(x,\phi,t)$ (bottom rows), respectively, on a cylindrical domain representing a real blood vessel. The figure is organized into three blocks based on the diffusion regimes, fast ($p=1.5$, upper block), normal ($p=2.0$, middle block), and slow ($p=3.0$, lower block). Base model with volume-filling factor is employed and other parameters are consistent with \cref{fig:2d_vf_p}. The concentrations are visualized as radial deformations from a baseline radius $R_0$, with initial conditions $P(x,\phi,0)=1.0-0.3\cos(2\pi x)\cos(\phi)$ and $W(x,\phi,0)=e^5$ over the cylindrical surface defined by axial length $x \in [0,1]$ and azimuthal angle $\phi \in [0,2\pi)$, shown at four representative times ($t=0,\,3.75,\,7.5,\,15$).}
\label{fig:2d_vf_c}
\end{figure}

The geometry of the initial condition of the cell density dictates the symmetry of the cell aggregation and chemical profiles illustrated in \cref{fig:2d_b_p}, \cref{fig:2d_b_p}, \cref{fig:2d_vf_p}, and \cref{fig:2d_vf_c}. In 1D, $P(x,0)=1.0-0.3\cos(2\pi x)$, whereas the 2D simulations employ $P(x,y,0)=1.0-0.3\cos(2\pi x)\cos(2\pi y)$ for planar and $P(x,\phi,0) = 1.0 - 0.3\cos(2\pi x)\cos(\phi)$ for cylindrical domains. Since the governing PDEs are isotropic, this initial symmetry is conserved. And the concentration profiles of both the cell density $P$ and the chemical field $W$ remain symmetric at every time step regardless of the diffusion regime (fast, normal, or slow).

\section{Energy Estimates and Analysis}\label{energy}

\subsection{Model Formulation}
We consider \cref{eq:logistic} on $\Omega\subset\mathbb{R}^n$ ($n = 1,2$), a bounded domain with smooth boundary, to obtain the following model with  $p>1$,
\begin{equation}
\begin{cases}
P_t = D\,\nabla\cdot\Big[(|\nabla P|^{p-2}+\theta)\nabla P - P\,\dfrac{\delta}{(W+\beta)(W+\gamma)}\,\nabla W\Big] + rP\Big(1-\dfrac{P}{K}\Big),\\[6pt]
W_t = D_1\Delta W + \Big(\dfrac{P}{1+\nu W}-\mu\Big)W,
\end{cases}
\label{eq:main_model}
\end{equation}
(for $x\in\Omega,\ t>0$) subject to the Neumann-type boundary conditions for $x\in\partial\Omega$
\begin{equation}
\Big[(|\nabla P|^{p-2}+\theta)\nabla P - P\,\frac{\delta}{(W+\beta)(W+\gamma)}\,\nabla W\Big]\cdot\nu = 0,
\qquad
\partial_\nu W = 0.
\label{eq:main_model_bcd}
\end{equation}

We state our main result,

\begin{theorem}
\label{thm:w1s}
    Consider \eqref{eq:main_model}. Let $2>p>\frac{3}{2}$, 
    and all other parameters be positive. If the initial data $P_{0}(x) \in C^{\omega}(\Omega), 0 < \omega <1, W_{0}(x) \in C^{2}(\Omega)$, is sufficiently small, $||P_{0}||_{\infty} < C, ||W_{0}||_{\infty} < C$, where $C$ depends on the parameters in the problem, then there exists a globally bounded weak solution to \eqref{eq:main_model}.
\end{theorem}

Our strategy to prove the above is to consider a regularised system and derive uniform estimetes in the regularising parameter $\epsilon$. Next we will show that the requisite convergences hold as $\epsilon \rightarrow 0$. For the special case $p=2$, we show there actually exists classical solution, and this is possible even without diffusion in the VEGF.
\subsection{Functional Preliminaries}

\begin{definition}[Weak solution]
Let
\[
0\le P\in L^1_{loc}(\bar\Omega\times[0,\infty)),
\]
\[
0\le W\in L^\infty_{loc}(\bar\Omega\times[0,\infty))\cap L^1_{loc}\big([0,\infty);W^{1,1}(\Omega)\big),
\]
and
\[
(|\nabla P|^{p-2}+\theta)|\nabla P|,\quad P\,|\nabla W|\ \in\ L^1_{loc}(\bar\Omega\times[0,\infty)).
\]
We call $(P,W)$ a global weak solution of \eqref{eq:main_model}--\eqref{eq:main_model_bcd} if the following equalities hold for all $\varphi\in C_0^\infty(\bar\Omega\times[0,\infty))$:
\begin{align*}
-\int_0^\infty\!\!\int_\Omega P\varphi_t\,dx\,dt-\int_\Omega P_0\varphi(\cdot,0)\,dx
={}& -D\int_0^\infty\!\!\int_\Omega (|\nabla P|^{p-2}+\theta)\nabla P\cdot\nabla\varphi\,dx\,dt\\
&+D\int_0^\infty\!\!\int_\Omega \frac{\delta\,P}{(W+\beta)(W+\gamma)}\,\nabla W\cdot\nabla\varphi\,dx\,dt\\
&+\int_0^\infty\!\!\int_\Omega rP\Big(1-\frac{P}{K}\Big)\varphi\,dx\,dt,
\end{align*}
and
\begin{align*}
-\int_0^\infty\!\!\int_\Omega W\varphi_t\,dx\,dt-\int_\Omega W_0\varphi(\cdot,0)\,dx
={}& -D_1\int_0^\infty\!\!\int_\Omega \nabla W\cdot\nabla\varphi\,dx\,dt\\
&+\int_0^\infty\!\!\int_\Omega \Big(\frac{P}{1+\nu W}-\mu\Big)W\,\varphi\,dx\,dt.
\end{align*}
\end{definition}

\bigskip

\begin{lemma}\label{lem:2.1prime}
Let $\Omega\subset\mathbb{R}^n$ ($n\ge 1$) be a bounded domain with smooth boundary, $1\le p,q\le\infty$, and $D_1,\mu>0$.

\begin{enumerate}
\item[(i)] If $\dfrac{n}{2}\Big(\dfrac1p-\dfrac1q\Big)<1$, then there exists $C=C(D_1,\mu,n,p,q,\Omega)>0$ such that whenever
$\zeta\in C^{2,1}(\bar\Omega\times(0,T))\cap C^0(\bar\Omega\times[0,T))$ for $T\in(0,\infty]$ solves
\[
\begin{cases}
\zeta_t=D_1\Delta\zeta-\mu\zeta+g, & (x,t)\in\Omega\times(0,T),\\
\partial_\nu\zeta=0, & (x,t)\in\partial\Omega\times(0,T),\\
\zeta(x,0)=\zeta_0(x), & x\in\Omega,
\end{cases}
\]
with $g\in C^0(\bar\Omega\times[0,T))$ and $\zeta_0\in W^{1,\infty}(\Omega)$, it holds that
\[
\|\zeta(\cdot,t)\|_{L^q(\Omega)} \le C\Big(\sup_{s\in(0,t)}\|g(\cdot,s)\|_{L^p(\Omega)}+\|\zeta_0\|_{L^\infty(\Omega)}\Big)
\qquad\text{for each } t\in(0,T).
\]

\item[(ii)] Assume that $\dfrac12+\dfrac{n}{2}\Big(\dfrac1p-\dfrac1q\Big)<1$. Then there exists $\tilde C=\tilde C(D_1,\mu,n,p,q,\Omega)>0$ such that if
$\zeta\in C^{2,1}(\bar\Omega\times(0,T))\cap C^0(\bar\Omega\times[0,T))$ with $T\in(0,\infty]$ is a solution of the above problem, then
\[
\|\nabla\zeta(\cdot,t)\|_{L^q(\Omega)} \le \tilde C\Big(\sup_{s\in(0,t)}\|g(\cdot,s)\|_{L^p(\Omega)}+\|\nabla\zeta_0\|_{L^\infty(\Omega)}\Big)
\qquad\text{for any } t\in(0,T).
\]
\end{enumerate}
\end{lemma}

\begin{lemma}
\label{lem:wlp}
Consider exponents $0 < p_0 < p_1 < \infty$ and a domain $\Omega$ that is a closed and bounded in $\mathbb{R}^{n}$, $n\geq 1$, and a $u(x) \in L^{p_1}(\Omega)$. Then,

\begin{equation}
||u||_{L^{p_{\theta}}(\Omega)} \leq C||u||^{1-\theta}_{L^{p_{0}}(\Omega)} ||u||^{\theta}_{L^{p_{1}}(\Omega)}
\end{equation}

for all $0 \leq \theta \leq 1$, and where $p_{\theta}$ is defined as $\frac{1}{p_{\theta}} = \frac{1-\theta}{p_0} + \frac{\theta}{p_1}$.
\end{lemma}

\begin{lemma}
\label{lem:gns}
Consider a $\phi \in W^{m,q^{'}}(\Omega) \cap L^{q}(\Omega)$. Then if $p^{'}, q^{'}, q \geq 1, 0 \leq \theta \leq 1$, and

\begin{equation}
k - \frac{n}{p^{'}} \leq \theta \left(  m - \frac{n}{q^{'}} \right) - (1 - \theta) \frac{n}{q},
\end{equation}
there exists a constant C such that,

\begin{equation}
	||\phi||_{W^{k,p^{'}}(\Omega)} \leq C ||\phi||^{\theta}_{W^{m,q^{'}}(\Omega)} ||\phi||^{1 - \theta}_{L^{q}(\Omega)}.
	\end{equation}

\end{lemma}	

Now we state the classical Aubin-Lions compactness Lemma,

\begin{lemma}
\label{lem:al}
Let $X_0, X$ and $X_1$ be three reflexive Banach spaces with 

$X_0 \hookrightarrow \hookrightarrow X \hookrightarrow X_1$. That is $ X_0$ is compactly embedded in $X$ and that $X$ is continuously embedded in $X_1$. For $1\leq p,q \leq \infty$, let

 \begin{equation}
W = \{ u \in L^{p^{'}}([0,T];X_0) \ | \  u^{'} \in L^{q^{'}}([0,T];X_1)  \}
\end{equation}

(i) If $p^{'} < \infty$ then the embedding of $W$ into $L^{p^{'}}([0,T];X)$ is compact.

(ii) If $p^{'} = \infty$ and $q^{'} > 1$ then the embedding of $W$ into $C([0,T];X)$ is compact.
\end{lemma}

\begin{lemma}
\label{lem:pc1}
    Consider functions $\phi, \sigma \in C^{1}(\Omega), \Omega \subset \mathbb{R}^{n}, n=1,2, |\Omega| < \infty$, $\epsilon > 0$. Then if $p \geq 2$ we have,
    \begin{equation}
        (|\nabla \phi|^{2} + \epsilon)^{\frac{p-2}{2}}|\nabla \phi|^{2} \geq |\nabla \phi|^{p},
\end{equation}

\begin{equation}
        (|\nabla \phi|^{2} + \epsilon)^{\frac{p-2}{2}}|\nabla \phi| \leq C\left(|\nabla \phi|^{p}+1\right),
\end{equation}

and 

\begin{equation}
        (|\phi|^{p-2}\phi - |\sigma|^{p-2}\sigma)\cdot (\phi - \sigma) \geq C |\phi - \sigma|^{p}.
\end{equation}

\end{lemma}

\begin{lemma}
\label{lem:pc1aa}
    Consider functions $\phi, \sigma \in C^{1}(\Omega), \Omega \subset \mathbb{R}^{n}, n=1,2, |\Omega| < \infty$, $\epsilon > 0$. Then if $1< p <2$ we have,
    \begin{equation}
        (|\nabla \phi|^{2} + \epsilon)^{\frac{p-2}{2}}|\nabla \phi|^{2} \geq \left(|\nabla \phi|^{2}+\epsilon \right)^{\frac{p}{2}} - (\epsilon)^{\frac{p}{2}}.
\end{equation}

\end{lemma}

\subsection{Regularised System}

We consider first the more difficult case $\theta=0$,

\begin{equation}\label{eq:ext3dr}
\begin{cases}
\displaystyle
(P_{\epsilon})_t = D\,\nabla \cdot
      \Bigl[\bigl(|\nabla P_{\epsilon}|^{2} + \epsilon \bigr)^{\frac{p-2}{p}}\nabla P_{\epsilon}
            - P_{\epsilon}\,\frac{\delta}{(W_{\epsilon}+\beta)(W_{\epsilon}+\gamma)}\, \nabla W_{\epsilon} \Bigr]
      + rP_{\epsilon}\ \left(1-\frac{P_{\epsilon}}{K}\right),
\\[10pt]
\displaystyle
(W_{\epsilon})_t = D_{1} \Delta W_{\epsilon} + \Bigl(\frac{P_{\epsilon}}{1+\nu W_{\epsilon}}-\mu\Bigr)W_{\epsilon},
\end{cases}
\end{equation}

\begin{equation}
\Bigl[|\nabla P_{\epsilon}|^{2} + \epsilon \bigr)^{\frac{p-2}{p}}\nabla P_{\epsilon}
            - P_{\epsilon}\,\frac{\delta}{(W_{\epsilon}+\beta)(W_{\epsilon}+\gamma)}\, \nabla W_{\epsilon} \Bigr] \cdot \nu = 0 \ , 
            \nabla W_{\epsilon} \cdot \nu =0.
\end{equation}
\begin{remark}
    We require in general the source term on the $P_{\epsilon}$ equation to be of the form, $f(P_{\epsilon}) \leq b P_{\epsilon} - \mu (P_{\epsilon})^{r}, r\geq 1, b,\mu>0$.
\end{remark}

\begin{remark}
    In all of the estimates, $C, C_{i}$ are generic constants and can change in value from line to line, and within the same line if so required.
\end{remark}
\begin{lemma}
Consider \eqref{eq:ext3dr}, then we have that,
\begin{equation}
    \int_{\Omega} P_{\epsilon} dx \leq 
    \max{\left(\int_{\Omega}P_{0}dx, C \right)} 
    \end{equation}

\begin{equation}
   \int^{t}_{t-1} \int_{\Omega} P_{\epsilon} dx ds  \leq 
     C_{2} 
    \end{equation}
   
    \end{lemma}

 \begin{proof}
    The estimates follow via integartion of the equation for $P_{\epsilon}$ in space, followed by time integartion, and Gronwall's lemma.
        \end{proof}
Here $C$ depends on problem parameters, and the measure of the domain $|\Omega|$.

        \begin{lemma}
Consider \eqref{eq:ext3dr}, then we have that,
\begin{eqnarray}
   && \frac{d}{dt}\left(\int_{\Omega} (P_{\epsilon} +1)^{m}dx\right)+Dm(m-1) \int_{\Omega} (1+P_{\epsilon})^{m-2}|\nabla P_{\epsilon}|^{p} dx  \nonumber \\ 
   && \leq \int_{\Omega} (1+P_{\epsilon})^{m-2}\frac{\delta P_{\epsilon}}{(W_{\epsilon}+\beta)(W_{\epsilon}+\gamma)} \nabla P_{\epsilon} \nabla W_{\epsilon}dx + m(m-1)(1+(\epsilon)^{\frac{p}{2}})\int_{\Omega} (1+P_{\epsilon})^{m-2}dx \nonumber \\ 
   && + m(r)\int_{\Omega} (1+P_{\epsilon})^{m}dx
   -\frac{rm}{K}\int_{\Omega} (1+P_{\epsilon})^{m-1}(P_{\epsilon})^{2} dx .\nonumber \\
        \end{eqnarray}
\end{lemma}
        \begin{proof}
           Consider first the special case r=2. We multiply the first by $m((P_{\epsilon} +1)^{m-1}$, and integrate by parts to obtain

\begin{eqnarray}
   && \frac{d}{dt}\left(\int_{\Omega} (P_{\epsilon} +1)^{m}dx\right)+Dm(m-1) \int_{\Omega} (1+P_{\epsilon})^{m-2}|\nabla P_{\epsilon}|^{p} dx  \nonumber \\ 
   && \leq \int_{\Omega} (1+P_{\epsilon})^{m-2}\frac{\delta P_{\epsilon}}{(W_{\epsilon}+\beta)(W_{\epsilon}+\gamma)} \nabla P_{\epsilon} \nabla W_{\epsilon}dx + m(m-1)\int_{\Omega} (1+P_{\epsilon})^{m-2}dx \nonumber \\ 
   && + m(1+\mu)\int_{\Omega} (1+P_{\epsilon})^{m-1}dx -2^{1-r}\int_{\Omega} (1+P_{\epsilon})^{m+r-1} dx \nonumber \\
   &&\leq \int_{\Omega} \frac{\delta}{\beta \gamma}(1+P_{\epsilon})^{m-1} \nabla P_{\epsilon} \nabla W_{\epsilon}dx + m(m-1)(1+(\epsilon)^{\frac{p}{2}})\int_{\Omega} (1+P_{\epsilon})^{m-2}dx \nonumber \\ 
   && + m(r)\int_{\Omega} (1+P_{\epsilon})^{m}dx -\frac{rm}{K}\int_{\Omega} (1+P_{\epsilon})^{m-1}(P_{\epsilon})^{2} dx. \nonumber \\
        \end{eqnarray}

            This follows via positivity of the solutions, and the fact that $P_{\epsilon}<1+P_{\epsilon}$.
        \end{proof}

\begin{lemma}
\label{lem:lpb}
Consider \eqref{eq:ext3dr}, then we have that for $p > \frac{3}{2}$, and all other parameters positive, and for initial data $P_{0}(x) \in C^{\omega}(\Omega), 0 < \omega <1, W_{0}(x) \in C^{2}(\Omega)$, chosen sufficiently small $||P_{0}||_{\infty} < C, ||W_{0}||_{\infty}<C$, for any $m>1$, we have,
\begin{equation}
    \int_{\Omega} (1+P_{\epsilon})^{m} dx \leq 
    C.
    \end{equation}
Here $C$ is a pure constant that depends only on the problem parameters.
   
    \end{lemma}

\begin{proof}

Now we use Lemma's \ref{lem:pc1}-\ref{lem:pc1aa}, and whether $1<p<2$ or $p>2$, we obtain,

\begin{eqnarray}
   && \frac{d}{dt}\left(\int_{\Omega} (P_{\epsilon} +1)^{m}dx\right)+\frac{D}{2}m(m-1) \int_{\Omega} (1+P_{\epsilon})^{m-2}|\nabla P_{\epsilon}|^{p} dx  \nonumber \\ 
&\leq& \frac{\delta}{\beta \gamma}\int_{\Omega} (1+P_{\epsilon})^{m-1} \nabla P_{\epsilon} \nabla W_{\epsilon}dx  + m(m-1)(1+(\epsilon)^{\frac{p}{2}})\int_{\Omega} (1+P_{\epsilon})^{m-2}dx \nonumber \\ 
   && + m(r)\int_{\Omega} (1+P_{\epsilon})^{m}dx -\frac{rm}{K}\int_{\Omega} (1+P_{\epsilon})^{m-1}(P_{\epsilon})^{2} dx. \nonumber  \\
        \end{eqnarray}

We also have 
\begin{equation}
\label{eq:mm1}
   \frac{1}{2} \int_{\Omega} (1+P_{\epsilon})^{m-1}2((P_{\epsilon})^{2}+1) dx \nonumber \\ 
\geq \frac{1}{2} \int_{\Omega} (1+P_{\epsilon})^{m-1}(1+P_{\epsilon})^{2}dx,\nonumber 
        \end{equation}
thus 

\begin{eqnarray}
\label{eq:11e}
   && \frac{d}{dt}\left(\int_{\Omega} (P_{\epsilon} +1)^{m}dx\right)+\frac{D}{2}m(m-1) \int_{\Omega} (1+P_{\epsilon})^{m-2}|\nabla P_{\epsilon}|^{p} dx  
   + \frac{rm}{2K} \int_{\Omega} (1+P_{\epsilon})^{m+1}dx\nonumber \\ 
&\leq& \frac{\delta}{\beta \gamma}\int_{\Omega} (1+P_{\epsilon})^{m-1} \nabla P_{\epsilon} \nabla W_{\epsilon}dx  + m(m-1)(1+(\epsilon)^{\frac{p}{2}})\int_{\Omega} (1+P_{\epsilon})^{m-2}dx \nonumber \\ 
   && + m(r)\int_{\Omega} (1+P_{\epsilon})^{m}dx 
   +\frac{rm}{K} \int_{\Omega} (1+P_{\epsilon})^{m-1}dx.
   \nonumber  \\
        \end{eqnarray}

        We state the following lemma,
        \begin{lemma}
        \label{lem:pe21}
            Consider \eqref{eq:ext3dr}. Consider $2>p>\frac{3}{2}$, then we have that 
\begin{equation}
 \int (1+P_{\epsilon})^{m-1}dx \leq  \left(\int (1+P_{\epsilon})^{4m-2}dx\right)^{\frac{1}{4}} \leq \left( \frac{2m-1}{3}\right)C(m) \int_{\Omega} (1+P_{\epsilon})^{m-2}|\nabla P_{\epsilon}|^{p} dx.
        \end{equation}
\begin{proof}
    The proof follows via the embedding of $W^{1,p}(\Omega) \hookrightarrow \hookrightarrow L^{6}(\Omega)$, for $n=2$, $p>3/2$. Also via, $(\int_{\Omega}fdx)^{4} \leq |\Omega|^{3}\int_{\Omega}f^{4}dx$.
\end{proof}

        \end{lemma}
Note via Young's inequality with $\epsilon$ we have,

\begin{eqnarray}
   && \frac{\delta}{\beta \gamma}\int_{\Omega} (1+P_{\epsilon})^{m-1} \nabla P_{\epsilon} \nabla W_{\epsilon}dx  \nonumber \\ 
&\leq& \frac{Dm(m-1)}{2}\int_{\Omega} (1+P_{\epsilon})^{m-2}|\nabla P_{\epsilon}|^{p} dx + C_{3}\int_{\Omega} (1+P_{\epsilon})^{\left(m+\frac{p}{p-1}-2\right)}|\nabla W_{\epsilon}|^{\frac{p}{p-1}} dx. \nonumber \\
        \end{eqnarray}

Here, $\epsilon^{p}=\frac{\beta \gamma}{\delta}\frac{Dm(m-1)}{2}$, which dictates that $C_{3}=\left(\frac{2 \delta}{\beta \gamma Dm(m-1)}\right)^{\frac{1}{p-1}}$.

Inserting the above estimate into \eqref{eq:11e}, and via Holder-Young's inequality, and lemma \ref{lem:pe21} above, we have,

\begin{eqnarray}
   && \frac{d}{dt}\left(\int_{\Omega} (P_{\epsilon} +1)^{m}dx\right)+\frac{Dm(m-1)}{4} \int_{\Omega} (1+P_{\epsilon})^{m-2}|\nabla P_{\epsilon}|^{p} dx  \nonumber \\ 
&\leq&  \left(\frac{2 \delta\theta_{1}}{\beta \gamma D}\right)^{\frac{1}{p-1}}\int_{\Omega} |\nabla W_{\epsilon}|^{\frac{p}{p-1}\theta_{1}} dx + C_{6}\int_{\Omega} (1+P_{\epsilon})^{\left(m+\frac{p}{p-1}-2\right)\left(\frac{\theta_{1}}{\theta_{1}-1}\right)} + m(r)\int_{\Omega} (1+P_{\epsilon})^{m}dx \nonumber \\
&& -\frac{rm}{K}\int_{\Omega} (1+P_{\epsilon})^{m-1}(P_{\epsilon})^{2} dx.  \nonumber \\
        \end{eqnarray}

Here, $\theta_{1}>1$. Note, via Neumann semi-group,


\begin{eqnarray}
\label{eq:wwp}
   && \int_{\Omega} |\nabla W_{\epsilon}|^{\frac{p\theta_{1}}{p-1}} dx  \nonumber \\ 
&\leq&  C_{5}||\nabla W_{\epsilon}||^{\frac{p\theta_{1}}{p-1}}_{\frac{p\theta_{1}}{p-1}} \leq C_{6}||(-\Delta + 1)W_{\epsilon}||^{\frac{p\theta_{1}}{p-1}(a_{2})}_{L^{m}(\Omega)}||W_{\epsilon}||^{\frac{p\theta_{1}}{p-1}(1-a_{2})}_{L^{q}(\Omega)} \leq C_{7}||(-\Delta + 1)W_{\epsilon}||^{\frac{p\theta_{1}}{p-1}(a_{2})}_{L^{m}(\Omega)}.\nonumber \\
 \end{eqnarray}

Next via lemma \ref{lem:2.1prime}, we have,

\begin{equation}
    \int^{t}_{t-1}||(-\Delta + 1)W_{\epsilon}||^{\frac{p\theta_{1}}{p-1}(a_{2})}_{L^{m}(\Omega)}ds \leq \sup_{\tau \in (0,t)}\left(||W_{\epsilon}||^{m}_{L^{m}(\Omega)}\right)^{\frac{p\theta_{1}}{p-1}},
\end{equation}
the mean value theorem for integrals gives the existence of a time $t^{*}$ s.t,

\begin{equation}
    (t-(t-1)||(-\Delta + 1)W_{\epsilon}(t^{*})||^{\frac{p\theta_{1}}{p-1}(a_{2})}_{L^{m}(\Omega)}ds \leq \sup_{\tau \in (0,t)}\left(||P_{\epsilon}||^{m}_{L^{m}(\Omega)}\right)^{\frac{p\theta_{1}}{p-1}},
\end{equation}

since t is arbitrary in $(0,T^{\epsilon}_{max})$, we have that for any $t$ in this time interval,

\begin{eqnarray}
   && \frac{d}{dt}\left(\int_{\Omega} (P_{\epsilon} +1)^{m}dx\right)+ \frac{Crm}{K}\int_{\Omega} (1+P_{\epsilon})^{m+1} dx \nonumber \\ 
&\leq&  \left(\frac{2 (m^{p-1})\delta\theta_{1}}{\beta \gamma D}\right)^{\frac{1}{p-1}}\left(\int_{\Omega} (1+P_{\epsilon})^{m} dx \right)^{\frac{p\theta_{1}}{p-1}} + C_{6}\int_{\Omega} (1+P_{\epsilon})^{\left(m+\frac{p}{p-1}-2\right)\left(\frac{\theta_{1}}{\theta_{1}-1}\right)} \nonumber \\
&& + m(r)\int_{\Omega} (1+P_{\epsilon})^{m}dx . \nonumber \\
        \end{eqnarray}

This enables us to compare to the ODE, for large $m$,
$\frac{dX}{dt} =- X + C_{1}X^{\frac{p\theta_{1}}{p-1}} $.

Choosing $\theta_{1}=2, p=\frac{3}{2}$, we obtain $C_{1} = (\frac{4 \delta}{\beta \gamma D})^{2}$,
Now, uniform boundedness of $X$, is ensured as long as the initial data satisfies, $X(0)=\int_{\Omega} (1+P_{\epsilon}(0))^{m} dx \leq \frac{1}{C_{1}}$. Taking supremum yields,
\begin{equation}
\boxed{||P_{\epsilon}(0)||_{\infty} \leq \left(\frac{\beta \gamma D}{4 \delta}\right)^{2}.}
\end{equation}
\end{proof}

\subsection{Global Existence of Weak Solutions}

In this section, we prove global existence of weak solutions to \eqref{eq:main_model}. The strategy is to consider  system \eqref{eq:ext3dr}, and derive uniform apriori estimates (in the parameter $\epsilon$). Next we pass to the limit as $\epsilon \rightarrow 0$. We begin by stating the following lemma,

\begin{lemma}
\label{lem:ge11n}
    Consider \eqref{eq:ext3dr}. Let $p>\frac{3}{2}$, and all other parameters be positive. If the initial data $P_{0}(x) \in C^{\omega}(\Omega), 0 < \omega <1, W_{0}(x) \in C^{2}(\Omega)$, is sufficiently small $||P_{0}||_{\infty} < C, ||W_{0}||_{\infty}<C$, then for any $1>> \epsilon > 0$, there exists a constant $C$, depending only on the problem parameters, such that we have,

\begin{equation}
\label{eq:11l}
    ||P^{\epsilon}||_{L^{\infty}(\Omega)} \leq C,
\end{equation}

\begin{equation}
\label{eq:12l}
    ||W^{\epsilon}||_{W^{1,\infty}(\Omega)} \leq C,
\end{equation}
and
\begin{equation}
\label{eq:14l}
\int^{t}_{0}\int_{\Omega} 
|\nabla P^{\epsilon}(.,s)|^{p} dx ds \leq C.
\end{equation}
\end{lemma}
\begin{proof}
    \eqref{eq:11l}-\eqref{eq:12l} follow from the estimates in Lemma \ref{lem:lpb}, followed via standard Moser-Alikakos scheme \cite{zhuang2021global}. Furthermore, we have the following inequality, after applying the estimates in Lemma \ref{lem:lpb}, with $m=2$ yields,

    \begin{eqnarray}
   && \frac{d}{dt}\left(\int_{\Omega} (P_{\epsilon} +1)^{2}dx\right)+\frac{D}{2} \int_{\Omega} |\nabla P_{\epsilon}|^{p} dx  \nonumber \\ 
&\leq&  \left(\frac{2 \delta\theta_{1}}{\beta \gamma D}\right)^{\frac{1}{p-1}}\int_{\Omega} |\nabla W_{\epsilon}|^{\frac{p}{p-1}\theta_{1}} dx + C_{6}\int_{\Omega} (1+P_{\epsilon})^{\left(\frac{p}{p-1}\right)\left(\frac{\theta_{1}}{\theta_{1}-1}\right)} + 2(r)\int_{\Omega} (1+P_{\epsilon})^{2}dx \nonumber \\
&& -\frac{2r}{K}\int_{\Omega} (1+P_{\epsilon})(P_{\epsilon})^{2} dx.  \nonumber \\
        \end{eqnarray}
    Now following the methods of estimate \eqref{eq:mm1}, we have,

    \begin{eqnarray}
&&\frac{d}{dt}\left(\int_{\Omega} (P_{\epsilon} +1)^{2}dx\right)+\frac{D}{2} \int_{\Omega} |\nabla P_{\epsilon}|^{p} dx \nonumber \\
&&\leq \left(\frac{2 \delta\theta_{1}}{\beta \gamma D}\right)^{\frac{1}{p-1}}\int_{\Omega} |\nabla W_{\epsilon}|^{\frac{p}{p-1}\theta_{1}} dx +C_{6}\int_{\Omega} (1+P_{\epsilon})^{\left(\frac{p}{p-1}\right)\left(\frac{\theta_{1}}{\theta_{1}-1}\right)} + C_{1}. \nonumber \\
        \end{eqnarray}
        Now we use lemma \ref{lem:lpb} and eatimates via \ref{eq:wwp}, and  integration in time $[0,T]$ of the inequality above, gives the integral estimate \eqref{eq:14l}.
    \end{proof}

Next, we state a result about the time derivative of the solution. We use the embedding,

        \begin{equation}
            W^{1,p}(\Omega) \hookrightarrow \hookrightarrow L^{q^{*}}(\Omega) \hookrightarrow (W^{1,p}(\Omega))^{*},
        \end{equation}

        and apply the $(W^{1,p}(\Omega))^{*}=W^{-1,q}(\Omega)$, where $\frac{1}{p} + \frac{1}{q} = 1$. 

\begin{theorem}
\label{thm:w1si}
    Consider \eqref{eq:ext3dr}. Let $2>p>\frac{3}{2}$, 
    and all other parameters be positive. If the initial data $P_{0}(x) \in W^{1,\infty}(\Omega), W_{0}(x) \in L^{\infty}(\Omega)$, is sufficiently small $||P_{0}||_{\infty} < C, ||W_{0}||_{\infty}<C$, then for any $1>> \epsilon > 0$, there exists a constant $C(T)$, depending only on the problem parameters, such that,

    \begin{equation} 
\left \lVert \frac{\partial P^{\epsilon}}{\partial t}\right \rVert_{L^{1}((0,T); (W^{1,p}(\Omega))^{*})} \leq C.
\end{equation}

\end{theorem}

\begin{proof}
We consider a test function $\zeta \in C^{\infty}_{0}(\Omega)$ such that $||\zeta||_{W^{1,p}(\Omega)} \leq 1$. Note that,

 \begin{equation} 
\left \lVert \frac{\partial P^{\epsilon}}{\partial t}\right \rVert_{L^{1}((0,T); (W^{1,p}(\Omega))^{*})} = \int^{T}_{0}\left( \sup_{\zeta \in C^{\infty}_{0}(\Omega),||\zeta||_{W^{1,p}(\Omega)} \leq 1 } \int_{\Omega} \frac{\partial P^{\epsilon}}{\partial t} \zeta dx\right) dt.
\end{equation}
We now consider,

\begin{eqnarray}
    && \int_{\Omega} \frac{\partial P^{\epsilon}}{\partial t} \zeta dx \nonumber \\
    && = \int_{\Omega} \Bigl[\nabla \cdot \Big(D_{1}  \left(|\nabla P^{\epsilon}|^{2}+{\epsilon}\right)^{\frac{p-2}{2}}  \nabla P^{\epsilon} - P_{\epsilon}\,\frac{\delta}{(W_{\epsilon}+\beta)(W_{\epsilon}+\gamma)}(\nabla W_{\epsilon})\Bigr]\zeta dx \nonumber \\
     && + r\int_{\Omega}P_{\epsilon}\ \left(1-\frac{P_{\epsilon}}{K}\right) \zeta dx \nonumber \\
    && \leq \int_{\Omega} |\nabla \zeta| |P^{\epsilon}||\nabla W_{\epsilon}| dx + r\int_{\Omega} | \zeta| |P^{\epsilon}| dx + \int_{\Omega} \frac{r}{K}| \zeta| |P^{\epsilon}|^{2} dx  \nonumber \\
    && + \int_{\Omega}\Big(D_{1}  \left(|\nabla P^{\epsilon}|^{2}+{\epsilon}\right)^{\frac{p-2}{2}}  \nabla P^{\epsilon} \Big) \nabla \zeta dx \nonumber \\
    &&  \leq \int_{\Omega} |\nabla \zeta| |P^{\epsilon}| dx + r\int_{\Omega} | \zeta| |P^{\epsilon}| dx + \frac{r}{K}\int_{\Omega} | \zeta| |P^{\epsilon}|^{2} dx  \nonumber \\
    &&+ \int_{\Omega}\Big(D_{1]}  \left(|\nabla P^{\epsilon}|^{2}+{\epsilon}\right)^{\frac{p-1}{2}}  \Big) \nabla \zeta dx \nonumber \\
    && \leq C_{1} + C_{2} + C_{3} + C_{4} + C_{5}.\nonumber \\
\end{eqnarray}
This follows via the uniform estimates in Lemma \ref{lem:ge11n}, as well as the assumptions on the test function $\zeta$.
\end{proof}

We state the following theorem,
\begin{theorem}
\label{thm:vubd}
    Consider \eqref{eq:ext3dr}. Let $p>\frac{3}{2}$, and all other parameters positive. Consider initial data $P_{0}(x) \in W^{1,\infty}(\Omega), W_{0}(x) \in L^{\infty}(\Omega)$, sufficiently small. Then there exists a function $P \in L^{p}_{loc}((0,\infty), W^{1,p}(\Omega)) \cap L^{\infty}((0,\infty);L^{\infty}(\Omega))$ and 
    
    $W \in  L^{2}_{loc}((0,\infty), W^{1,2}(\Omega)) \cap L^{\infty}((0,\infty);L^{\infty}(\Omega))$, and a $\Gamma \in L^{\frac{p}{p-1}}_{loc}((0,\infty);L^{\frac{p}{p-1}}$ $(\Omega)) \cap L^{\infty}((0,\infty);L^{\infty}(\Omega))$, and a sequence of approximants $\epsilon = \epsilon_{j} \searrow 0$, such that,

    \begin{equation}
P^{\epsilon} \rightarrow P^{*} \ \mbox{in} \ L^{p}_{Loc}((0,\infty);L^{p}(\Omega)),
\end{equation}

 \begin{equation}
P^{\epsilon} \overset{*}{\rightharpoonup} P^{*}  \ \mbox{in} \ L^{\infty}((0,\infty);L^{\infty}(\Omega)),
\end{equation}

 \begin{equation}
 \label{eq:w1pv}
\nabla P^{\epsilon} \rightharpoonup \nabla P^{*}  \ \mbox{in} \ L^{p}_{loc}((0,\infty);L^{p}(\Omega)),
\end{equation}

 \begin{equation}
|\nabla P^{\epsilon}|^{p-2}\nabla P^{\epsilon} \rightharpoonup  \Gamma \ \mbox{in} \ L^{\frac{p}{p-1}}_{loc}((0,\infty);L^{\frac{p}{p-1}}(\Omega)),
\end{equation}

\begin{equation}
\nabla W^{\epsilon} \overset{*}{\rightharpoonup} \nabla W^{*}  \ \mbox{in} \ L^{\infty}((0,\infty);L^{\infty}(\Omega)),
\end{equation}
and
\begin{equation}
 W^{\epsilon} \rightharpoonup  W^{*}  \ \mbox{in} \ L^{2}_{loc}((0,\infty);W^{1,2}(\Omega)).
\end{equation}

\end{theorem}

\begin{proof}
    The proof follows via the uniform bounds in Lemma \ref{lem:ge11n}, Theorem \ref{thm:w1si} and the Aubin-Lions compactness theory via Lemma \ref{lem:al}.
\end{proof}
Next standard methods \cite{zhuang2021global, parshad2026effect} show, $\Gamma = |\nabla P|^{p-2}\nabla P$, whereby the proof of Theorem \ref{thm:w1s} is complete.   
\subsection{Finite time extinction of the cell}

\begin{theorem}
\label{thm:FFTEdd}
Consider the model \eqref{eq:main_model}. Then there exists certain positive initial data $(P_0(x),W_0(x))$, and some $p \in(\frac{3}{2},2]$, such that  $ P \to 0$ in $L^{2}(\Omega)$, in finite time.
\end{theorem}

\begin{proof}

We consider $n=1$, and test the $P$ equation in $(\ref{eq:main_model})$, against $P$ itself. Next, integrating over the full domain $\Omega$ and using the boundary conditions we obtain,
\[ \dfrac{1}{2} \dfrac{d}{dt} ||P||_2^2 +  \frac{r}{K}||P||_3^3 +   D_{1}||P_x||_p^p \le r||P||_2^2 + \frac{\delta}{\beta \gamma}\int_{\Omega} P_x W_{x} P dx .\]
Using the positivity of $P,W$ and earlier estimates, we have,
\[ \dfrac{1}{2} \dfrac{d}{dt} ||P||_2^2 +  ||P||_3^3 + C_{1}||P_x||_p^p \le  ||P||_2^2 + ||P||_{2 + \frac{2-p}{p-1}}^{2 + \frac{2-p}{p-1}} + ||W_{x}||^{q}_{q},\]

We first use the weighted $L^{p}$ inequality via lemma \ref{lem:wlp}

\[  L^{3}(\Omega) \hookrightarrow L^{2 + \frac{2-p}{p-1}}(\Omega) \hookrightarrow L^{2} (\Omega),\]
followed by Young's inequality
to obtain,
\begin{equation}
    ||P||_{2 + \frac{2-p}{p-1}} \leq C_{1}||P||^{2}_{3} + C_{2}||P||^{2}_{2} \leq C_{3} + C_{4}||P||^{3}_{3} + C_{2}||P||^{2}_{2},
\end{equation}
such that,

\begin{equation}
    \dfrac{1}{2} \dfrac{d}{dt} ||P||_2^2 +  C_{5}||P||_3^3 + C ||Px||_p^p \le C_{4}||P||_2^2 + C_{3} .
\end{equation}

Now for $\theta=\frac{1}{2}$, we apply Lemma \ref{lem:gns} to obtain,

\begin{equation}\label{GNS_pde}
    ||P||_{2} \le C ||P_x||^{\theta}_{p} ||P||^{1-\theta}_{3}.
\end{equation}
We raise the both sides of $(\ref{GNS_pde})$ to the power of $l$, where $l \in (0,2)$
\[ \Big( \int_\Omega P^2 \Big)^{\frac{l}{2}} \le C \Big( \int_\Omega |P_x|^{p}dx \Big)^{\frac{l \theta}{p}} \Big( \int_\Omega |P|^3dx \Big)^{\frac{l(1-\theta)}{3}}. \]
Note via Young's inequality \cite{evans},
 $ab \le \dfrac{a^r}{r} + \dfrac{b^s}{s}$
such that $\frac{1}{r} + \frac{1}{s}=1.$ 
Use the Young's inequality for $r = \frac{p}{l \theta}$ and $s=\frac{p}{l(1-\theta)}$, where $\theta = \frac{1}{2}$.  Moreover, $p \in (1,2]$, so we can find a $l \in (0,2)$ such that,
\begin{equation}\label{ODE}
    Y_t \le C_{3} + MY-\widetilde{C} Y^{\alpha},
\end{equation}
where $Y=||P||_2$. As $p\in (1,2]$, we can fix $\alpha =\frac{p}{2}\in (0,1).$ Now via \cite{parshad2021some}, we can prove that $\exists T^*<\infty$ such that $Y \to 0$ as $T \to T^*$.
\end{proof}

\subsection{Special case p=2}

In this setting our system reduces to,

\begin{equation}
\begin{cases}
P_t = D\,\nabla\cdot\Big[\nabla P - P\,\dfrac{\delta}{(W+\beta)(W+\gamma)}\,\nabla W\Big] + rP\Big(1-\dfrac{P}{K}\Big), & x\in\Omega,\ t>0,\\[6pt]
W_t = D_1\Delta W + \Big(\dfrac{P}{1+\nu W}-\mu\Big)W, & x\in\Omega,\ t>0,
\end{cases}
\label{eq:main_model1n}
\end{equation}
subject to the Neumann-type boundary conditions
\begin{equation}
\Big[\nabla P - P\, \chi(W)\,\nabla W\Big]\cdot\nu = 0,
\qquad
\partial_\nu W = 0,
\qquad x\in\partial\Omega,\ t>0,
\label{eq:main_model_bcdn1}
\end{equation}
where $\chi(W) = \frac{\delta}{(W+\beta)(W+\gamma)}$.
Herein via standard results, \cite{mizukami2016global}, a classical soulution is expected. We apply the methods in \cite{mizukami2016global}, to derive an inequality of the form,

\begin{equation}
\frac{d}{dt}\int_{\Omega} P^{p}\left[ f(W)\right]^{-l} \leq a \int_{\Omega} P^{p}\left[ f(W)\right]^{-l} - b \left( \int_{\Omega} P^{p}\left[ f(W)\right]^{-l}\right)^{\frac{p+1}{p}},
    \end{equation}
where $f(W) = e^{\left(\int^{W}_{0}\chi(s)ds\right)}$.

\begin{lemma}
\label{lem:pf1}
Consider \eqref{eq:main_model}, then there exists a $r=r(D_{1},p)>0$ such that,
\begin{eqnarray}
&&\frac{d}{dt}\int_{\Omega} P^{p}\left[ f(W)\right]^{-l} \nonumber \\ 
&\leq& p r \int_{\Omega} P^{p}\left[ f(W)\right]^{-l}(1-P)dx - r \int_{\Omega}P^{p}\left[ f(W)\right]^{-l}(\chi(W)\Big(\dfrac{P}{1+\nu W}- \mu\Big)Wdx. \nonumber \\
\end{eqnarray}
\end{lemma}

\begin{proof}
    We begin with the following inequality for,

    \begin{eqnarray}
&& \frac{d}{dt}\int_{\Omega} P^{p}\left[ f(W)\right]^{-l} \nonumber \\
&\leq& p\int_{\Omega}P^{p}\left[ f(W)\right]^{-l} \nabla \cdot (\nabla P - P \chi(W)\nabla W)dx + pr \int_{\Omega}P^{p}\left[ f(W)\right]^{-l}(1-P)dx \nonumber\\
&-& lD_{1} \int_{\Omega}P^{p}\left[ f(W)\right]^{-l}\chi(W)\Delta W dx \nonumber\\
&& - l\int_{\Omega}P^{p}\left[ f(W)\right]^{-l}\chi(W)(\frac{P}{1+\nu W} - \mu)Wdx. \nonumber\\
         \end{eqnarray}

Next we show that, under certain parametric choice, we have,

\begin{eqnarray}
&&I_{1} +I_{2} \nonumber \\
&& = 
    p\int_{\Omega}P^{p}\left[ f(W)\right]^{-l} \nabla \cdot (\nabla P - P \chi(W)\nabla W)dx 
- lD_{1} \int_{\Omega}P^{p}\left[ f(W)\right]^{-l}\chi(W)\Delta W dx \nonumber \\
&& \leq 0 \nonumber .\\
\end{eqnarray}
standard calculations following \cite{mizukami2016global} yield,

\begin{equation}
\label{eq:i12}
   I_{1} + I_{2} \leq \int_{\Omega}P^{p}\left[ f(W)\right]^{-l}(C_{1}(\chi(W))^{2}+D_{1}l\chi^{'}(W))|\nabla W|^{2}dx.
\end{equation}
    Note, we can choose l, s.t $C_{1}(\chi(W))^{2}+D_{1}l\chi^{'}(W) \leq 0$, thus showing $I_{1} + I_{2} < 0$, and thus the lemma is proved.
    \end{proof}

Now we use the above lemma to yield, 
    \begin{eqnarray}
&& \frac{d}{dt}\int_{\Omega} P^{p}\left[ f(W)\right]^{-l} \nonumber \\
&\leq&  pr \int_{\Omega}P^{p}\left[ f(W)\right]^{-l}(1-P)dx  - l\int_{\Omega}P^{p}\left[ f(W)\right]^{-l}\chi(W)(\frac{P}{1+\nu W} - \mu)Wdx \nonumber \\
&\leq& (pr+l\mu) \int_{\Omega}P^{p}\left[ f(W)\right]^{-l}-pr\int_{\Omega}P^{p+1}\left[ f(W)\right]^{-l}dx - l\int_{\Omega}P^{p+1}\left[ f(W)\right]^{-l}\frac{\chi(W)}{1+\nu W}dx\nonumber \\
&\leq& (pr+l\mu) \int_{\Omega}P^{p}\left[ f(W)\right]^{-l}-pr\int_{\Omega}P^{p+1}\left[ f(W)\right]^{-l}dx \nonumber \\
&\leq& (pr+l\mu) \int_{\Omega}P^{p}\left[ f(W)\right]^{-l} - C_{1} \left(\int_{\Omega}P^{p}\left[ f(W)\right]^{-l}\right)^{\frac{p+1}{p}} \nonumber .\\
         \end{eqnarray}

From this differential inequality it follows that

\begin{equation}
e^{-l||\chi{(W)}||_{L^{1(0,\infty)}}}
\int_{\Omega}P^{p} dx \leq \int_{\Omega}P^{p}\left[ f(W)\right]^{-l}\leq C.
\end{equation}

This enables the following lemma,

\begin{lemma}
\label{lem:Plp}
Consider \eqref{eq:main_model}. Then for a choice of parameters such that, we have,

\begin{equation}
||P||_{L^{p}(\Omega)} \leq Ce^{l||\chi{(W)}||_{L^{1(0,\infty)}}}.
\end{equation}

\end{lemma}

We now make some further estimates for $||W||_{W^{1,\infty}(\Omega)}$, which will enable a trivial estimate for the nonlinear term,
\begin{equation}
    ||P \chi(W) \nabla W||_{L^{p}(\Omega)} \leq C |\chi(W(0))|||P||_{L^{p}(\Omega)}||W||_{W^{1,\infty}(\Omega)}.
\end{equation}

Note, standard estimates via Neumann semi-group yields,

\begin{equation}
    || \nabla W||_{L^{\infty}(\Omega)} \leq C_{1}||(-\Delta + C)^{\rho}W||_{L^{p}(\Omega)}.
\end{equation}

This enables the following Theorem,

\begin{theorem}
\label{thm:w1s1}
    Consider \eqref{eq:main_model}. For positive parameters, and any initial data $P_{0}(x) \in C^{0}(\overline{\Omega}), W_{0}(x) \in W^{1,q}(\Omega)$, then there exists a global classical solution to \eqref{eq:main_model}.
\end{theorem}

\begin{remark}
We see that even for $D_{1}=0$, the estimates via \eqref{eq:i12} and lemma \ref{lem:Plp}, hold, thus the well posedness of our system without diffusion in the VEGF is also guaranteed. 
    \end{remark}

We now state certain propositions,

\begin{proposition}
    Consider \eqref{eq:main_model} with $p=2$, then for certain initial data and parameter choice, solutions can lead to shock formation. However, for $p>2$, solutions initiating from the same initial data and parameters do not form shocks.
\end{proposition}

\begin{proposition}
    Consider \eqref{eq:main_model} with $p=2$, then for certain initial data and parameter choice, solutions can lead to one spike dynamics. However, for $1<p<2$, solutions initiating from the same initial data and parameters can lead to multiple spike dynamics.
\end{proposition}

\section{Pattern-forming Instability}
\label{pattern}
Turing patterns in chemotaxis models have attracted considerable interest, \cite{painter2002volume,ma2016pattern,diele2025pattern}. We study Turing-type instability and pattern formation for the $p$-Laplacian chemotaxis--VEGF model with a logistic  term
for the cell density $P$:
\begin{equation}\label{eq:pat_base}
\begin{cases}
P_t = D_P\,\nabla\cdot\Big[(\theta_1|\nabla P|^{p-2}+\theta_2)\nabla P - P\,\dfrac{\delta}{(W+\beta)(W+\gamma)}\,\nabla W\Big] + r_1P\Big(1-\dfrac{P}{K}\Big),\\[6pt]
W_t = D_W\Delta W + r_2\Big(\dfrac{P}{1+\nu W}-\mu\Big)W,
\end{cases}
\end{equation}
on $\Omega$ (We take a bounded interval in 1D or a rectangle in 2D for convenience.)
with zero-flux boundary conditions \cref{eq:main_model_bcd}.

\subsection{Stability Analysis ($p>2$)}
\label{sec:experiments:stability}

Setting all spatial derivatives to zero in \cref{eq:pat_base} gives the coexistence homogeneous steady state
\begin{equation}\label{eq:pat_ss}
  P^*=K,\qquad W^*=\frac{K-\mu}{\mu\nu}\qquad(\text{requires } K>\mu),
\end{equation}

Consider the one-dimension case in space. Write $P=P^*+\widetilde P$, $W=W^*+\widetilde W$ with
$\widetilde P,\widetilde W\ll 1$. Since $P_x^*=0$, the degenerate term
$\theta_1|P_x|^{p-2}P_x$ vanishes at linear order whenever $p>2$: writing
$\phi(s):=|s|^{p-2}s$, one has
$\phi'(0)=(p-1)\lim_{s\to0}|s|^{p-2}=0$, so only the constant $\theta_2$
contributes an effective diffusion coefficient $D_P\theta_2$. The
chemotactic drift linearises to $D_PP^*\chi^*\widetilde W_x$ by Taylor's expansion, with
\begin{equation}
  \chi^*:=\chi(W^*)=\frac{\delta}{(W^*+\beta)(W^*+\gamma)}.
\end{equation}
Differentiating the kinetic terms
$f(P,W)=r_1P(K-P)$ and $g(P,W)=r_2\bigl(P/(1+\nu W)-\mu\bigr)W$
at $(P^*,W^*)$ gives
\begin{equation}\label{eq:pat_jacobian}
  f_P^*=-r_1K,\qquad f_W^*=0,\qquad
  g_P^*=\frac{r_2\mu W^*}{K},\qquad
  g_W^*=-\frac{r_2\mu(K-\mu)}{K}.
\end{equation}

Zero-flux boundary conditions admit normal modes
$\widetilde P=Ae^{\lambda t}\cos(kx)$, $\widetilde W=Be^{\lambda t}\cos(kx)$
with $k=n\pi/L$, $n\in\mathbb Z$. Substituting into the linearised system
gives the eigenvalue problem $M(k)(A,B)^{\mathsf T}=\lambda(A,B)^{\mathsf T}$,
\begin{equation}\label{eq:pat_M}
  M(k)=
  \begin{pmatrix}
    -D_P\theta_2k^2+f_P^* & D_PP^*\chi^*k^2+f_W^*\\[4pt]
    g_P^* & -D_Wk^2+g_W^*
  \end{pmatrix}.
\end{equation}
The characteristic polynomial (\textbf{dispersion relation}) $\lambda^2-\mathrm{tr}(k^2)\lambda+\det(k^2)=0$
has
\begin{align}
  \mathrm{tr}(k^2) &= -(D_P\theta_2+D_W)k^2+(f_P^*+g_W^*),\label{eq:pat_tr}\\
  \det(k^2) &= D_P\theta_2D_Wk^4
    -\bigl[D_P\theta_2g_W^*+D_Wf_P^*+g_P^*D_PP^*\chi^*\bigr]k^2
    +f_P^*g_W^*.\label{eq:pat_det}
\end{align}

\subsubsection{Instability condition}
Since $f_P^*g_W^*-f_W^*g_P^*=r_1r_2\,\frac{K\mu(K-\mu)}{K}>0$ and $f_P^*+g_W^*<0$ for all $r_1,r_2>0$,
the spatially homogeneous kinetics alone are always stable.
And the homogeneous steady state can only lose stability
through $\det(k^2)$ turning negative at some admissible $k^2>0$.

Write $\det(x)=Ax^2+Bx+C$ with $x:=k^2$, where
\begin{equation}
  A=D_P\theta_2D_W>0,\qquad
  C=f_P^*g_W^*>0,\qquad
  B=-\bigl[D_P\theta_2g_W^*+D_Wf_P^*+g_P^*D_PP^*\chi^*\bigr].
\end{equation}
Because $A,C>0$, the parabola $\det(x)$ is positive at $x=0$ and opens
upward. Consequently, if the parameters admit real positive roots
$0<x_1<x_2$ of $\det(x)=0$, then
\begin{equation}
  \det(k^2)<0 \iff x_1<k^2<x_2,
\end{equation}
i.e.\ the instability band is a finite window
$k\in(k_{\mathrm{lo}},k_{\mathrm{hi}})=(\sqrt{x_1},\sqrt{x_2})$ strictly
between two neutral (marginal) wavenumbers; the system is linearly
stable for $k<k_{\mathrm{lo}}$ and for $k>k_{\mathrm{hi}}$.

Two conditions are required for such roots to exist: $B<0$ and $B^2>4AC$. The true onset of instability is where the
parabola's minimum value first touches zero, i.e.\ $B^2=4AC$. Expanding these two instability conditions gives
\begin{equation}\label{eq:conds_explicit}
\begin{cases}
D_WK\,r_1 \;<\; \frac{D_P\mu}{K}\Big[KW^*\chi^*-\theta_2(K-\mu)\Big]\,r_2,\\[5pt]
\Big(D_WK\,r_1-Q\,r_2\Big)^2
  \;>\;
  4\,D_P\theta_2D_W\,\mu(K-\mu)\,r_1r_2,
\end{cases}
\end{equation}
where $Q=\frac{D_P\mu}{K}\Big[KW^*\chi^*-\theta_2(K-\mu)\Big]$. We perform tests to show the area of instability in \cref{sec4.2.1}. 

\begin{remark}[Two spatial dimensions]\label{remark_Turing_2D}
On a rectangular domain $[0,L_x]\times[0,L_y]$ with Neumann boundary
conditions, admissible normal modes are
$\cos(k_xx)\cos(k_yy)$ with $k_x=n\pi/L_x$, $k_y=m\pi/L_y$, i.e.\
eigenfunctions of the 2D Laplacian operator with eigenvalue
$-|\mathbf k|^2=-(k_x^2+k_y^2)$. Because every spatial operator in
\eqref{eq:pat_base} (the diffusion term, the chemotactic drift, and
$D_W\nabla^2W$) is isotropic, the linearised system in 2D is identical to
\eqref{eq:pat_M} with $k^2$ reinterpreted as $|\mathbf k|^2$.(Note that high-oder infinitesimal $|\nabla \widetilde P|^{p-2}\nabla \widetilde P\sim O(\varepsilon^{p-1})$ is omitted during linearisation, where $\epsilon$ represents the amplitude of $\widetilde P$.) The two instability conditions \eqref{eq:conds_explicit} are also unchanged between one and two spatial dimensions. Here the unstable set in $(k_x,k_y)$-space is an quarter-annulus
$\{(k_x,k_y):k_{\mathrm{lo}}^2<k_x^2+k_y^2<k_{\mathrm{hi}}^2; k_x,k_y\geq 0 \}$. Every
direction on the annulus grows at the same rate for a given
$|\mathbf k|$. In \cref{fig:2d_3pts}, it is seen that patterns with slow diffusion are more connected than ones with only normal diffusion. It is worth-noting that to get patterns which are biologically relevant (i.e., angiogenesis), we may require anisotropic diffusion in both the normal (Fickian) diffusion term and the 
p-Laplacian term.
\end{remark}

\subsection{Numerical Experiments}
\label{sec:experiments:numerical}
We now illustrate \S\ref{sec:experiments:stability} numerically.

\subsubsection{Instability Region of $(r_1,r_2)$ and Patterns}\label{sec4.2.1}
We firstly fix
all parameters except $(r_1,r_2)$ to show the instability region with respect to $(r_1,r_2)$. The chosen three representative points $(r_1,r_2)$ used
and other parameters are given in \cref{tab:pat_params}. For each of the three points in \cref{tab:pat_params}, the full
nonlinear system \eqref{eq:pat_base} was integrated from the
homogeneous steady state \eqref{eq:pat_ss} perturbed by band-limited
random noise, for both $\theta_1=0$ and
$\theta_1=0.1$.  \cref{fig:1d_instability_region} shows the instability region of $(r_1,r_2)$. And \cref{fig:1d_spacetime_all} and \cref{fig:logistic_1d_amplitude_profile} show patterns and amplitude evolutions at three points $(r_1,r_2)$. The wave number of the saturating pattern in \cref{fig:logistic_1d_amplitude_profile}(f) at the deep-inside point is consistent with the fastest-growing wavenumber in \cref{fig:1d_instability_region}(b).

\begin{table}[htbp]
\footnotesize
\caption{Parameter values and representative $(r_1,r_2)$ points used in
the numerical experiments of \S\ref{sec4.2.1}. The label represents the location of points with respect to instability region. ("Stable" is outside of the instability region; "Near boundary" and "Deep inside" are inside.)}
\label{tab:pat_params}
\begin{center}
\begin{tabular}{|c|c|c|c|c|c|c|c|c|}
\hline
$K$ & $\mu$ & $\nu$ & $\beta=\gamma$ & $\delta$ & $D_P$ & $\theta_2$ & $D_W$ & $p$ \\ \hline
$1$ & $0.5$ & $1$ & $0.1$ & $0.3$ & $0.25$ & $0.05$ & $0.05$ & $4$ \\ \hline
\end{tabular}
\qquad
\begin{tabular}{|l|c|c|}
\hline
Label & $r_1$ ($r_2=1$) & $\max_k\mathrm{Re}(\lambda)$ \\ \hline
Stable        & $1.00$ & $-0.250$  \\
Near boundary & $0.28$ & $0.0057$ \\
Deep inside   & $0.05$ & $0.166$  \\ \hline
\end{tabular}
\end{center}
\end{table}

\begin{figure}[htbp]
\centering
\includegraphics[width=0.75\textwidth]{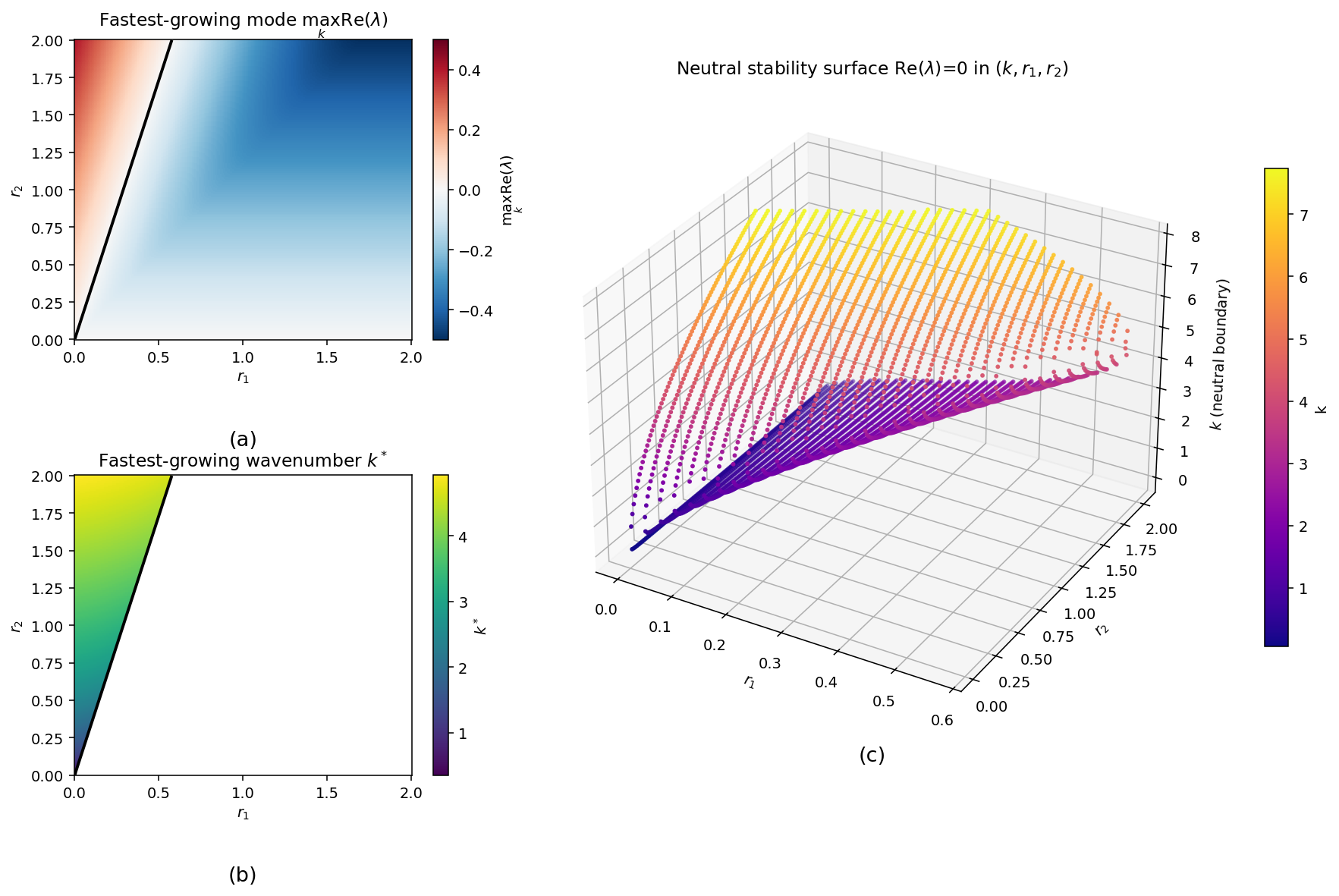}
\caption{(a)-(b) Instability region in $(r_1,r_2)$: the black contour is the neutral boundary of the instability region. Solving instability conditions \eqref{eq:conds_explicit} gives $r_1<\rho r_2$. The red area in (a) represents the instability region; the green area in (b) represents the fastest-growing wavenumber. (c) Neutral-stability surface $\mathrm{Re}(\lambda(k;r_1,r_2))=0$
in $(k,r_1,r_2)$-space; the two branches for each unstable $(r_1,r_2)$
are $k_{\mathrm{lo}}$ and $k_{\mathrm{hi}}$. Only points $(k,r_1,r_2)$ between these two branches have positive growth rates.}
\label{fig:1d_instability_region}
\end{figure}

\begin{figure}[htbp]
\centering
\includegraphics[width=0.95\textwidth]{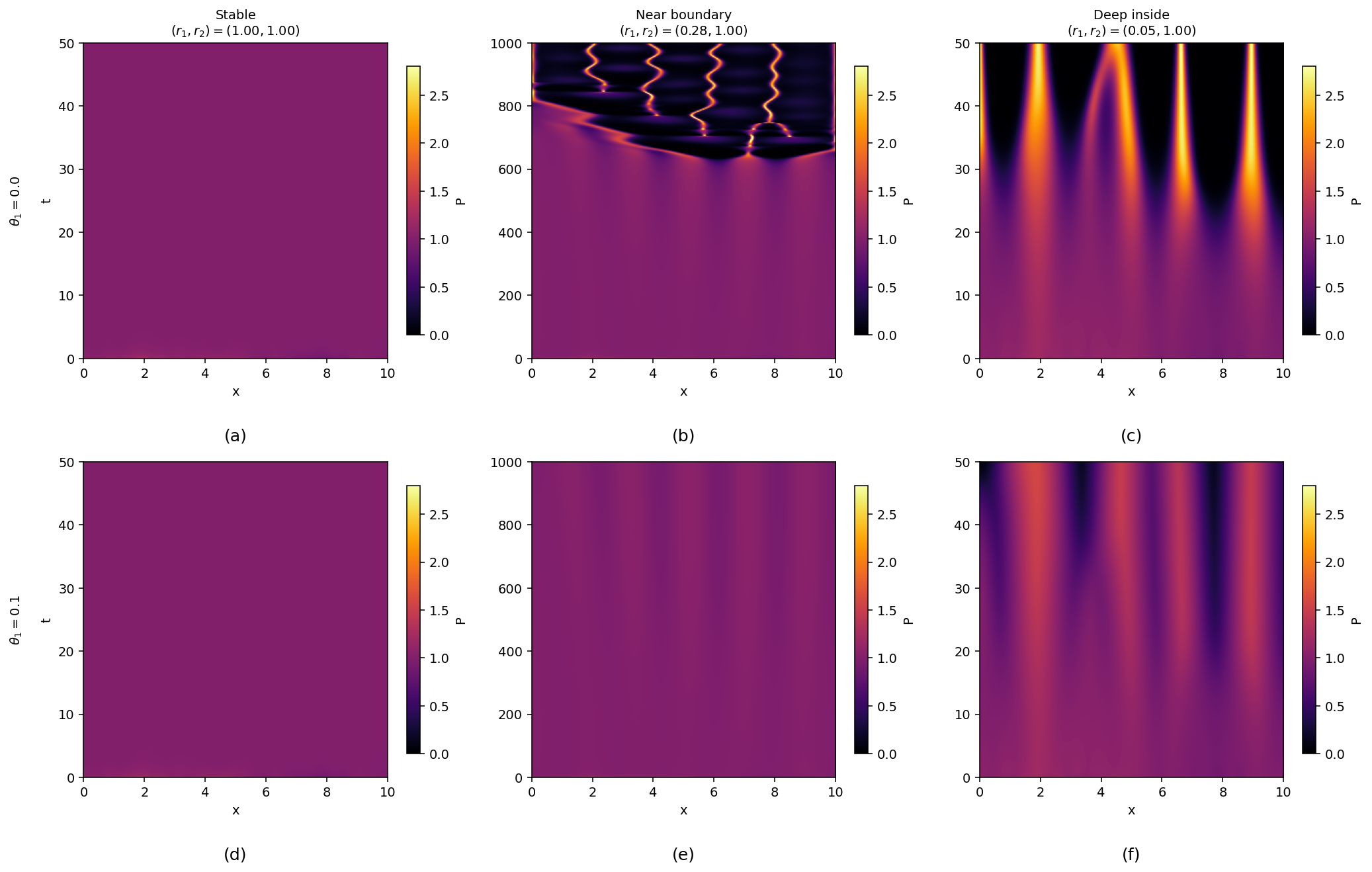}
\caption{Space--time evolution of $P(x,t)$: columns are the three
$(r_1,r_2)$ points of \cref{tab:pat_params}, rows are $\theta_1=0$ vs.\
$\theta_1=0.1$. The stable point decays to homogeneous in both rows,
confirming that $\theta_1$(p-Laplacian diffusion) does not alter the stability boundary; inside the unstable region, $\theta_1=0$
produces sharp and moving aggregation fronts, while
$\theta_1=0.1$ yields markedly smoother, bounded and stable patterns.}
\label{fig:1d_spacetime_all}
\end{figure}

\begin{figure}[htbp]
\centering
\includegraphics[width=0.80\textwidth]{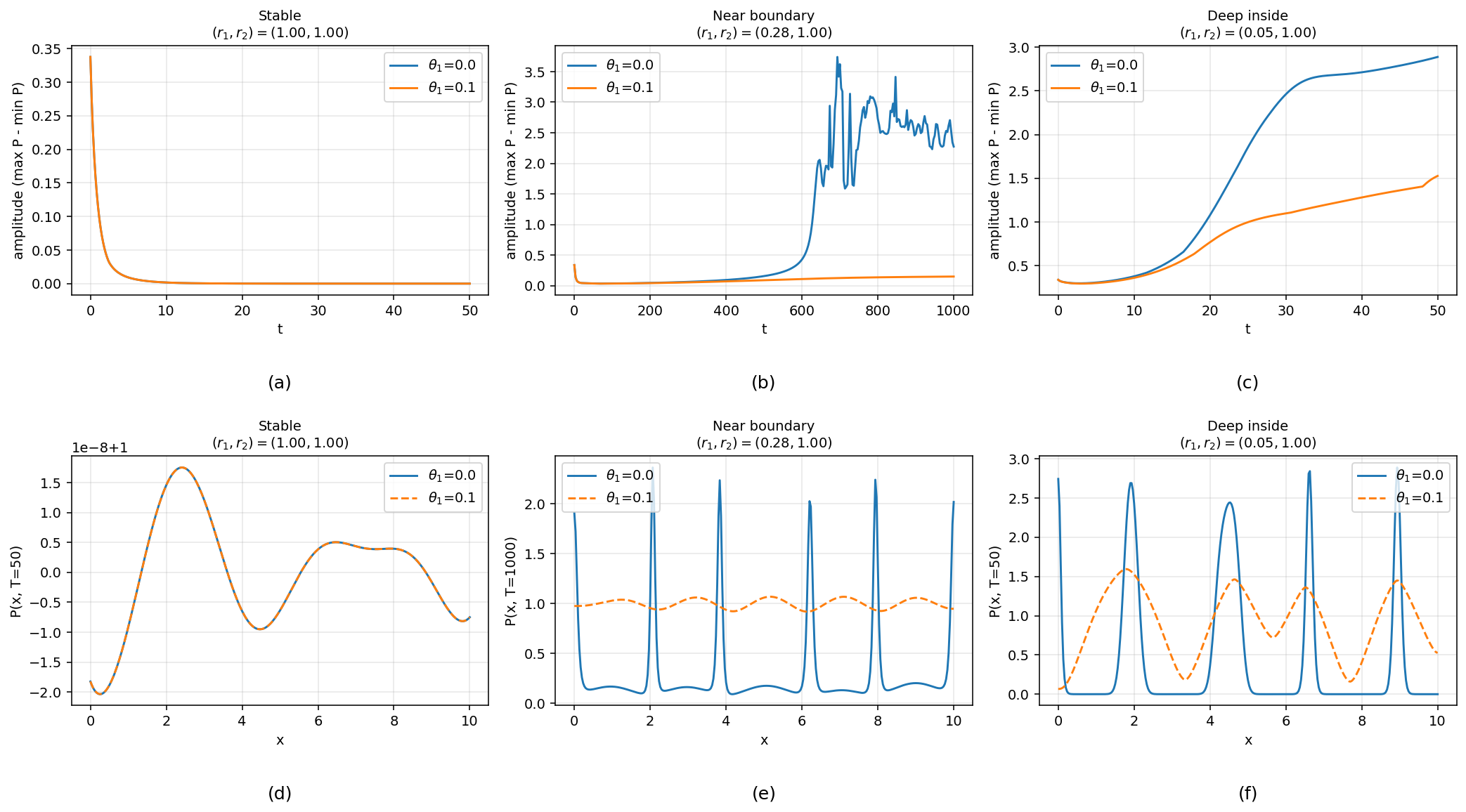}
\caption{(a-c) Amplitude ($\max P-\min P$) vs.\ time at the three
$(r_1,r_2)$ points. The near-boundary case
shows an initial transient decay (some fast-decaying modes excited by the
noise) followed by slow exponential growth and then oscillation during a long running time of 1000. And the deep-inside point shows a quick saturation of patterns at time 50. (d-f) Final spatial profiles $P(x,T)$: $\theta_1=0$ (solid) vs.\
$\theta_1=0.1$ (dashed).}
\label{fig:logistic_1d_amplitude_profile}
\end{figure}

In two spatial dimensions, the instability region is identical to 1D. Thus we repeat experiments in 2D for three representative points of \cref{tab:pat_params} as well. And \Cref{fig:2d_3pts} illustrates the quarter-annulus structure of the unstable set in $(k_x,k_y)$-space discussed in Remark \ref{remark_Turing_2D}. At points that are deep inside, simulations with $\theta_1=0.1$ gives rise to a connected pattern reminiscent of angiogenesis in \cref{fig:2d_3pts}(f), compared to $\theta_1=0$ which displays lots of spots. In our simulations, slow-diffusion could lead to smoother and more continuous patterns and prevent over-aggregation, and hence is more biologically meaningful.
\begin{figure}[htbp]
\centering
\includegraphics[width=0.80\textwidth]{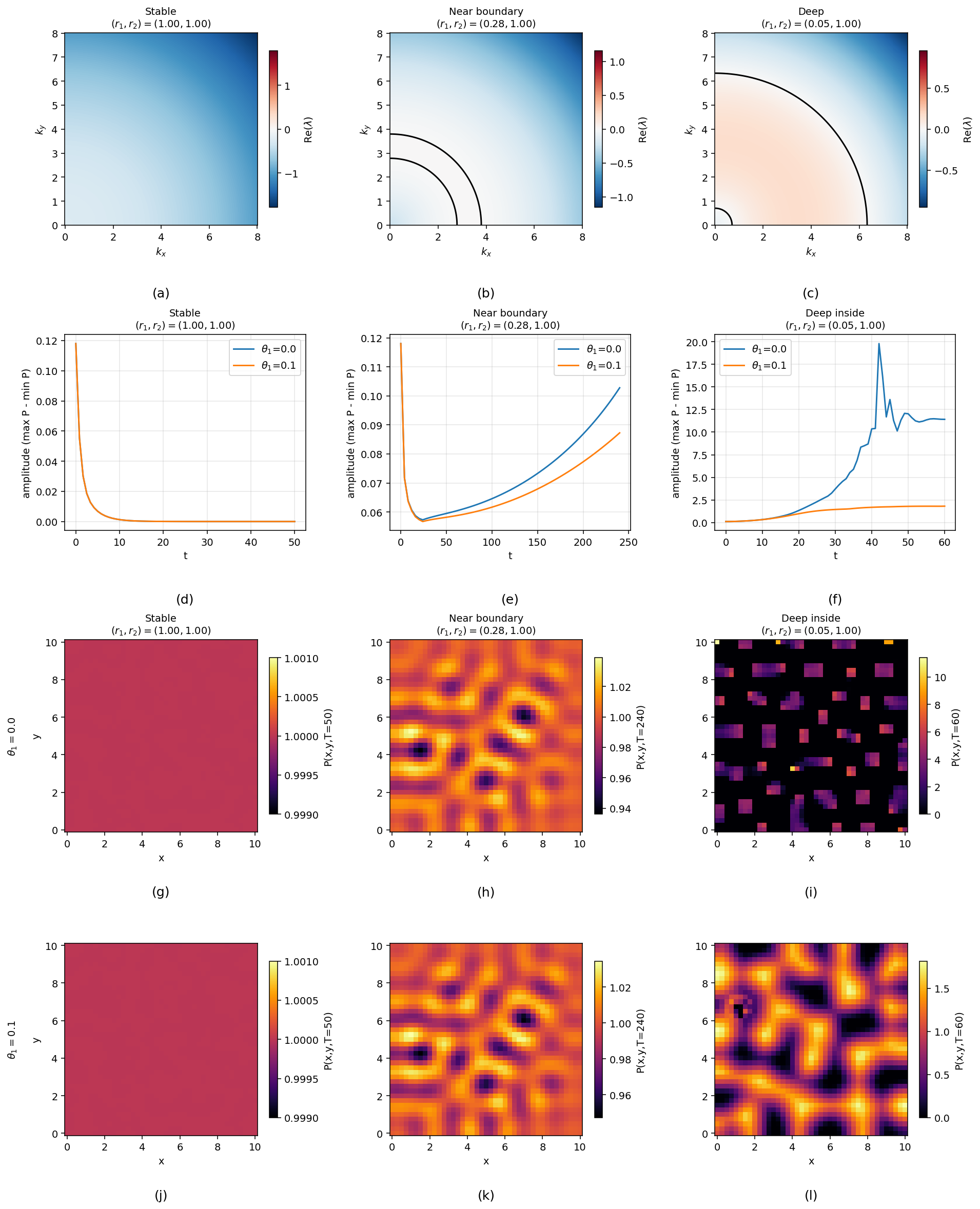}
\caption{(a-c) Unstable quarter-annulus in $(k_x,k_y)$-space for the three representative points in \cref{tab:pat_params}. The red area represents the unstable set. (d-f) 2D amplitude ($\max P-\min P$) vs.\ time. Evolution of 2D amplitudes resemble much as 1D in \cref{fig:logistic_1d_amplitude_profile}. (g-l) 2D simulations: final-time snapshots of $P(x,y,T)$
for the three points (columns) and $\theta_1=0$ vs.\ $\theta_1=0.1$
(rows). We run simulations until emergence of stable patterns.}
\label{fig:2d_3pts}
\end{figure}

\subsubsection{Effect of \texorpdfstring{$D_P,D_W$}{D\_P,D\_W} and Cell Kinetics}
Varying $D_P,D_W$ and fixing other parameters can change the fastest-growing wavenumber more directly than varying $(r_1,r_2)$, since these
coefficients also enter the $k^4$ terms of $\det(k^2)$ in
\eqref{eq:pat_det}.
\Cref{fig:pat_dpdw} shows four
$(D_P,D_W)$ combinations with slow diffusion (i.e. $\theta_1=0.1$). All four combinations evolve to bounded and well-connected patterns. The large diffusion on $W$ can lead to large wavelength and sparse network-like pattern (even coarse pattern). That is, strong random motility of VEGF does not attenuate chemotactic aggregation of $P$; rather, it renders the dynamics along the chemotactic direction more uniform, implying pattern of several dominant peaks (i.e. \Cref{fig:pat_dpdw}(b)) instead of multiple peaks (i.e. \Cref{fig:2d_3pts}(l)).

\begin{figure}[htbp]
\centering
\includegraphics[width=0.80\textwidth]{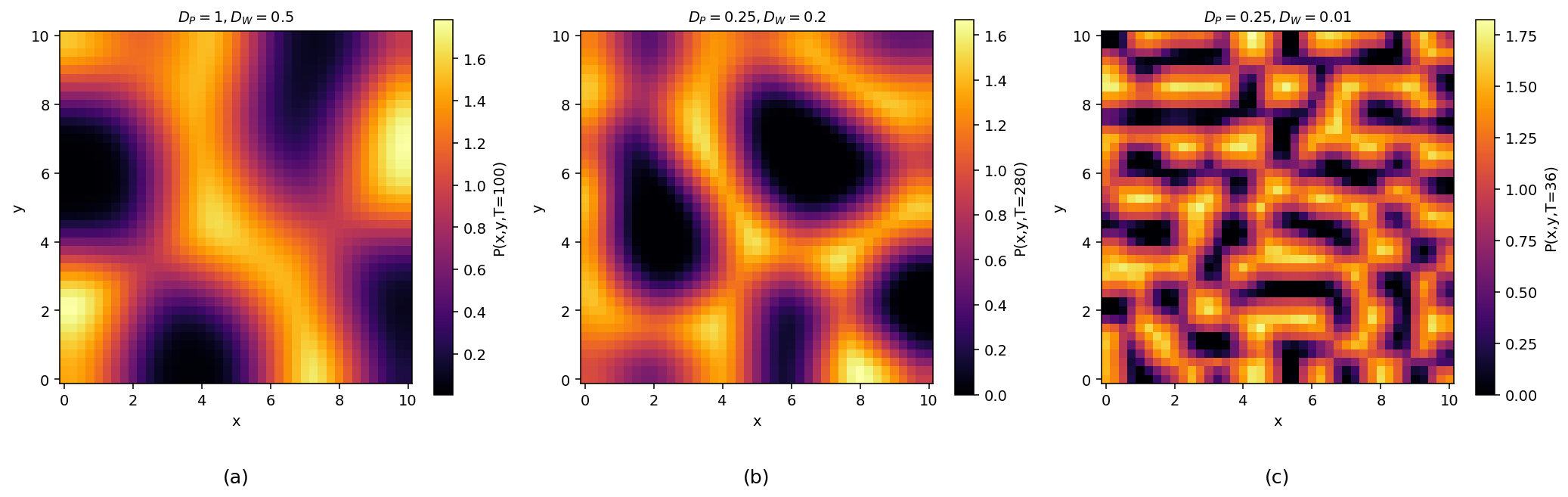}
\caption{2D Turing patterns: three combinations of $(D_P,D_W)$ with slow diffusion ($\theta_1=0.1$), and run simulations until emergence of stable patterns. Take $(r_1,r_2)$ as $(0.05,1.0)$, and other parameters align with \cref{tab:pat_params}. (a) $(D_P,D_W)=(1.0,0.5)$: Coarse pattern(one long tunnel connected with two short ones). (b) $(D_P,D_W)=(0.25,0.2)$: Multiple connected tunnels. (c) $(D_P,D_W)=(0.25,0.01)$: Labyrinth-like pattern.}
\label{fig:pat_dpdw}
\end{figure}

It is worth-noting that other cell kinetics, such as
$r_1WP(K-P)$($W$ dependent) and $r_1P(P-a)(K-P)$ (Allee-type, with $0<a<K$) are also reasonable. $r_1WP(K-P)$ indicates cell proliferation is also controlled by VEGF, which aligns with the fact that VEGF can stimulate growth of endothelial cells. And the Allee-type $r_1P(P-a)(K-P)$ incorporates the extreme case that local aggregation can deplete all the neighbor cells, due to its bistable
kinetics ($P=0$ stable, $P=a$ unstable, $P=K$ stable). However, numerical experiments (not shown here) indicate that these three cell kinetics do not differ a lot in saturating patterns. This implies that the chemotactic effect plays a dominant role in forming patterns, rather than cell kinetics.

\subsubsection{Patterns with $p<2$}
\label{sec:experiments:noiseamp}

The stability analysis of \S\ref{sec:experiments:stability} concerns
$p>2$, for which $\theta_1|P_x|^{p-2}$ vanishes at linear order about
the homogeneous state, so the instability conditions
are independent of any $\theta_1$. For $p<2$ this
is no longer true: $|P_x|^{p-2}$ \emph{diverges} as $P_x\to0$, so near the homogeneous background the $\theta_1$-term
behaves as an anomalously large diffusion coefficient, which greatly dampens the density $P$.

The same experiment, repeated on the 1D domain with the same
interior point in \cref{tab:pat_params}, $p=1.5$, $\theta_1\neq0$, and initial noise
amplitudes $\mathrm{amp}_0\in\{0.2,0.3,0.4,0.5,0.6\}$, shows that pattern emerges only when the noise amplitude is large enough in \cref{fig:fast_diffusion_1d_noiseamp_combined}. It is seen that the final profiles of $P$ share very similar bimodal aggregation behaviors as seen earlier with the one peak in the middle and another peak near the boundary. This indicates the form of noise controls saturating patterns, as long as the norm of initial $|P_x|$ exceed some critical value. The relationship between
initial noise level and shape of the forming pattern is therefore a threshold-like competition
between the singular small-gradient diffusion of $|P_x|^{p-2}$ and the underlying chemotactic instability.

\cref{fig:fast_diffusion_2d_noise_amplitude_combined} illustrates the 2D effect at the near-boundary point of \cref{tab:pat_params} with $p=1.5$ and $\theta_1=0.1$: band-limited noise of amplitude $\mathrm{amp}_0\in\{0.02,0.05,0.1,$ $0.15,0.2\}$ is
added to the homogeneous steady state and evolved to a fixed time $50$. And 2D simulations just align with the 1D case. The counterpart of a 2D annulus in \cref{fig:fast_diffusion_2d_noise_amplitude_combined}(b) is just the bimodal aggregation in 1D (See \ref{fig:fast_diffusion_1d_noiseamp_combined}(g)).

\begin{figure}[htbp]
\centering
\includegraphics[width=0.95\textwidth]{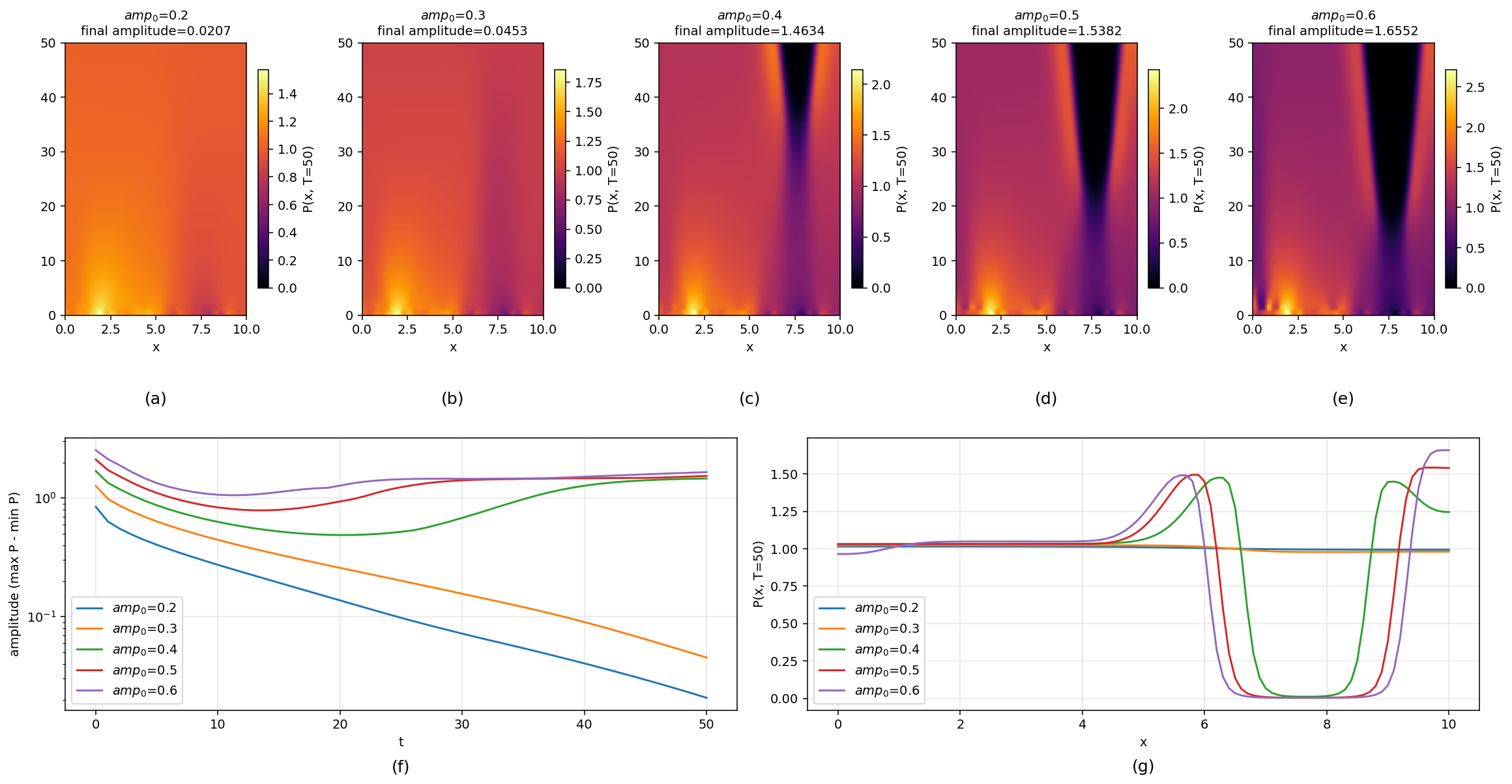}
\caption{(a-e) Space--time
evolution of $P(x,t)$ for the five initial noise amplitudes $amp_o$,
deep-inside point in \cref{tab:pat_params}, $p=1.5$, $\theta_1=0.1$. (f) Amplitude
($\max P-\min P$, log scale) vs.\ time for the five initial noise
levels. (g) Final profile of
$P$ for the five cases.}
\label{fig:fast_diffusion_1d_noiseamp_combined}
\end{figure}

\begin{figure}[htbp]
\centering
\includegraphics[width=0.80\textwidth]{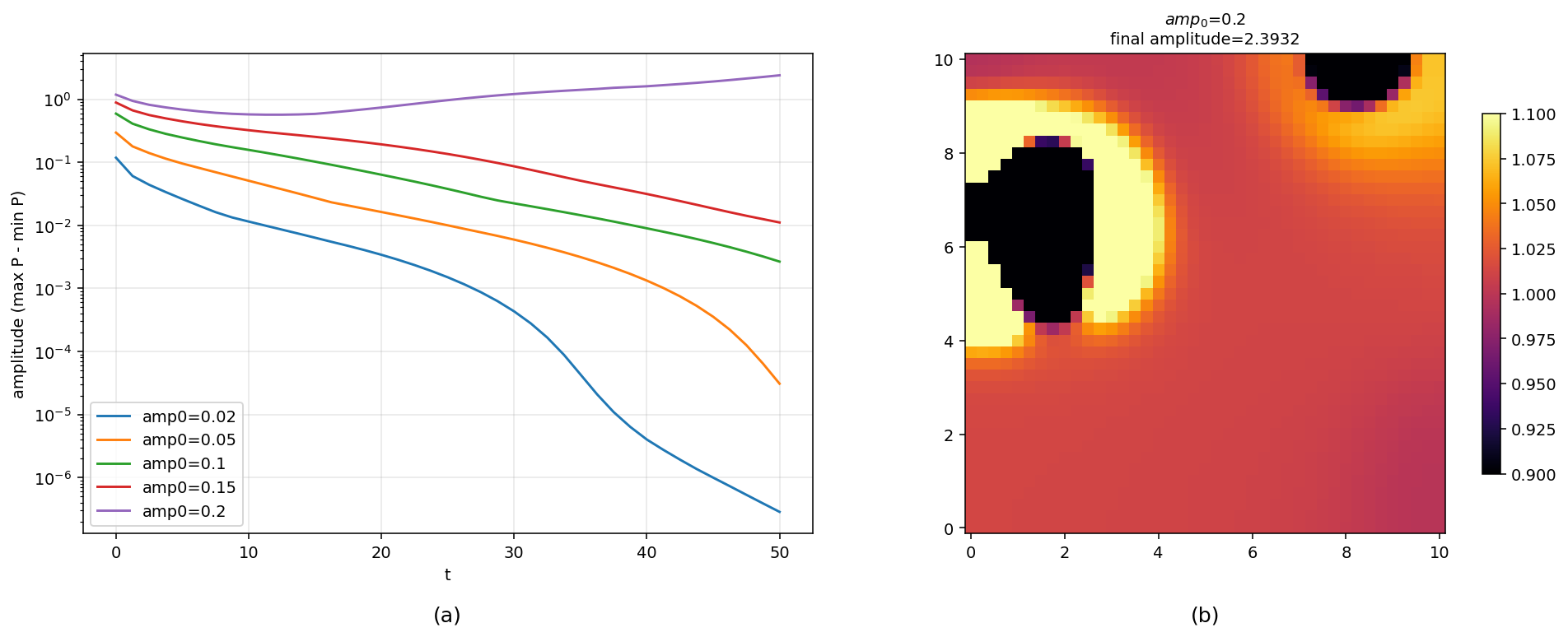}
\caption{(a) Amplitude ($\max P-\min P$, log scale) vs.\ time on 2D for five
initial noise amplitudes, deep-inside point in \cref{tab:pat_params}, $p=1.5$, $\theta_1=0.1$, and $T=50$.
The largest amplitude decays before reversing into growth; smaller
amplitudes decay for the entire simulation. (b) Final-time snapshot ($T=50$) for the initial noise
amplitude $\mathrm{amp}_0=0.2$.}
\label{fig:fast_diffusion_2d_noise_amplitude_combined}
\end{figure}

\section{Discussions}
\label{discussions}
Angiogenesis is a biological process that is key to tissue growth. Enhancing angiogenesis can assist biomedical applications to human disease, such as cancer and ischemic heart disease. Beyond the tumor setting, the mechanism analyzed here also underlies therapeutic angiogenesis in ischemic heart disease, where the goal is to promote rather than suppress collateral vessel formation; the present framework may help clarify which medical stimulation regimes
yield sustained, well-organized vascular networks (i.e. pattern forming and stabilizing). In the current manuscript we consider a system of PDE that models angiogenesis of cells, in response to VEGF.
The p-Laplacian motion we consider models both ``fast" and ``slow" movement, and can represent quick or slow growth, regression, regrowth and maturation of capillaries.
Our results on model \cref{eq:base_model} and its extension with logistic control \cref{eq:logistic} show novel rich dynamics in 1D and 2D spatial domains, that raise several questions for future research. The dynamics that we illustrate are distinct from other chemotaxis models to the best of our knowledge. For \cref{eq:logistic}, we show global existence of weak solution, for sufficiently small initial data. The linear analysis for instability conditions is explored as well. We find slow diffusion enhances the connection of tunnel-like structures in emerging patterns. Also, the bimodal behaviors due to fast diffusion perpetuates when gradients of initial distribution exceed some critical level (otherwise, the perturbations are dampened quickly). 

Based on our regularity results, there are two propositions for the dynamics that are observed:(i) The first posits the inhibition effect of slow diffusion to aggregation and shock formation - thus slow diffusion can prevent the initial distribution from forming shocks. 
(ii) The second posits bimodal aggregation from single-peak initial density distribution is possible due to fast diffusion - that is fast diffusion can increase cellular proliferation by developing bi-modal (as well as multi modal) spike dynamics. A rigorous proof of these dynamics is the object of future work. As concerns pattern formation, we observe that introduction of p-Laplacian has the potential to change the instability region, as compared to Fickian diffusion. This direction also requires further investigation.

The literature has not considered fast and slow diffusion, and the coupled effects of these with chemotaxis on the angiogenesis. We next mention several novel theoretical directions to be investigated with p-Laplacian diffusion, that is of signal depletion type \cite{lankeit2023review, arumugam2021keller}, in our future work - based on some of our simulations, that show blow-up prevention, via the p-Laplacian. In general we consider,

\begin{equation}
\label{eq:1m}
P_{t} = d_{1}\nabla \cdot (|\nabla P|^{p-2} \nabla P) -  \chi_{1}\nabla \cdot \left(P f_{1}(P)\nabla c\right) +  h_{1}(P), \ 
 W_{t}= d_{1}\Delta W - \lambda_{1}f(W)P + S(W)
\end{equation}

Here $1<p<\infty$. Also, this is equipped with appropriate initial and Neumann boundary conditions.
For logistic growth, ($h_{1}(P)=P(1-P/K)$), systems of type \eqref{eq:1m} possess classical solutions for all initial conditions in $n=2$, and weak solutions in $n \geq 3$ \cite{tello2007chemotaxis}. If one weakens control on $h_{1}$ (say $h_{1}(P)=P$), then finite time blow-up is possible, ($\lim_{t \rightarrow T^{*} < \infty} ||P||_{L^{\infty}(\Omega)} \rightarrow + \infty$), even in $n=2$ \cite{horstmann2005boundedness}. Preventing explosion of solutions or blow-up in such models has had quite a bit of recent activity \cite{surulescu2021does}. Damping mechanisms, have considered depletion terms typically of the form $f(W) = W^{q}, q\geq 1$. The case of ``fast" mechanisms ($0<q<1$, and $1<p<2$), in inhibiting blow up remains under explored. Akin to our recent work \cite{upadhyay2026eco, parshad2026effect}. Hence we conjecture,

\begin{conjecture}
    \label{con:cc}
    Consider \eqref{eq:1m} with $f(W)=W^{q}$, $h_{1}(P)=P$. Then for certain initial conditions $(P_{0},W_{0}) \in L^{\infty}(\Omega)$, and suitably chosen $0<p<1$, solutions remain bounded for all time. 
\end{conjecture}

In the classical setting, $f_{1}(P)=P^{\beta -1}$, and if $0 < \beta < \frac{2}{n}$, then solutions are global and bounded, and if $\beta \geq \frac{2}{n}, n \geq 2$, finite time blow-up can occur \cite{horstmann2005boundedness, horstmann20041970, levine1997siam}. We propose to investigate via the earlier conjecture, several novel methods to prevent blow-up.

The proposed mathematical model will establish the foundation for digital twin modeling in cardiac vascular regeneration.
Note, a ``digital twin" is a model of a physical system that uses real-time data to accurately reflect its real-world counterpart’s behavior. This mathematical model will predict the vasculature architecture with respect to the corresponding VEGF concentration field. This can be assimilated into a digital twin platform that will also combine  time-series imaging and genetic data into a dynamic, continuously updating model of angiogenesis. Unlike static or sequential models, the digital twin framework will couple multiscale simulations with time-series data, enabling predictive modeling of vascular trajectory and reverse engineering the precise genetic interventions required to achieve the targeted vasculature morphology. All in all our goal is to use the currently introduced class of models as a tool within the digital twin framework, to improve human cardiac health.


\section*{Acknowledgments}
WZ, HA, AP, and RP acknowledge valuable partial support from the National Science
Foundation (Grant No. DMS-2533961).

\bibliographystyle{siamplain}
\bibliography{references}
\end{document}